\documentclass[11pt,reqno]{amsart}  %
\usepackage{epsf}
\newtheorem{theorem}{Theorem}[section]
\newtheorem{proposition}[theorem]{Proposition}
\newtheorem{lemma}[theorem]{Lemma}
\newtheorem{corollary}[theorem]{Corollary}
\newtheorem{definition}[theorem]{Definition}
\newtheorem{remark}[theorem]{Remark}

\newcommand{\I}{\Delta^s_\infty}
\newcommand{\R}{\mathbb R}
\newcommand{\N}{\mathcal N}

\begin{document}
\title{Non-Local Tug-of-War and 
the Infinity Fractional Laplacian}
\author{C. Bjorland, L. Caffarelli, A. Figalli}

\begin{abstract} 
Motivated by the ``tug-of-war'' game studied in \cite{MR2449057},
we consider a ``non-local'' version of the game which goes as follows: at every step two players pick respectively a direction
and then, instead of flipping a coin in order to decide which direction to choose and then moving of a fixed amount $\epsilon>0$ (as is done in the classical case),
it is a $s$-stable Levy process which chooses at the same time both the direction and the distance to travel. Starting from this game,
we heuristically we derive a deterministic non-local integro-differential equation that we call ``infinity fractional Laplacian''. 
We study existence, uniqueness, and regularity, both for the Dirichlet problem and for a double obstacle problem, both problems having a natural interpretation as ``tug-of-war'' games.
\end{abstract}
\maketitle

\section{Introduction}
Recently Peres et al., \cite{MR2449057}, introduced and studied a class of two-player differential games called ``tug-of-war''.  Roughly, the game is played by two
players whose turns alternate based on a coin flip.  The game is played in some set $\Omega$ with a payoff function $f$ defined on $\partial\Omega$.
A token is initially placed at a point $x_0 \in \Omega$.
Then, on each turn, the player is allowed to move the token to any point in an open ball of size $\epsilon$ around the current position.
 If the players move takes the token to a point $x_f\in \partial\Omega$ then the game is stopped and the players are awarded or penalized by the payoff function $f(x_f)$.
In the limit $\epsilon\rightarrow 0$, the value function of this game is shown to solve the famous ``infinity Laplacian''
(see \cite{MR2408259} and the references therein). There are many variations on the rules of the game, for example adding a running cost of movement,
which give rise to a class of related ``Aronsson equations'', see \cite{MR2341994,MR2200259,MR2449057}.

In this paper we consider a variation of the game where, instead of flipping a coin, at each turn the players pick a direction and the distance moved in the chosen direction is determined by observing a stochastic process.
If the stochastic process is Brownian motion the corresponding limit of the value function will be the infinity Laplacian equation as before,
but if the stochastic process is a general Levy process the result will be a (deterministic) integro-differential equation with non-local behavior.
We call such a situation ``non-local tug-of-war'' and herein we study the case of a symmetric $s$-stable Levy process with $s\in (\frac{1}{2},1)$
(such processes are connected to the ``fractional Laplacian'' $(-\triangle)^s$). As we will show through a
heuristic argument in Section \ref{derivation}, this game will naturally lead to the following operator
(see also Subsection \ref{otherdefn}
for further considerations and a comparison with another possible definition of solution): 

\begin{definition}\label{IFLdefn}
For $s\in (\frac{1}{2},1)$ the ``infinity fractional Laplacian'' $\I:C^{1,1}(x)\cap BC(\mathbb{R}^N)\rightarrow \mathbb{R}$ at a point $x$ is defined in the following way: 
\begin{itemize}
\item If $\nabla\phi(x)\neq 0$ then
\begin{align}
\I\phi(x)=\int_0^\infty\frac{\phi(x+\eta v)+\phi(x-\eta v)-2\phi(x)}{\eta^{1+2s}}\,d\eta,\notag
\end{align}
where $v\in S^{N-1}$ is the direction of $\nabla\phi(x)$.
\item If $\nabla\phi(x)=0$ then
\begin{align}
\I\phi(x)= \sup_{y\in S^{N-1}}\int_0^\infty\frac{\phi(x+\eta y)-\phi(x)}{\eta^{1+2s}}\,d\eta+\inf_{z\in S^{N-1}}\int_0^\infty\frac{\phi(x-\eta z)-\phi(x)}{\eta^{1+2s}}\,d\eta.\notag
\end{align}
\end{itemize}
\end{definition}
In the above definition, $BC(\mathbb{R}^N)$ is used to denote the set of bounded continuous functions on $\mathbb{R}^N$, and functions which are $C^{1,1}$ at a point $x$
are defined in Definition \ref{C11defn}.

Given a domain $\Omega\subset \mathbb{R}^N$ and data $f:\R^N\setminus \Omega\rightarrow \mathbb{R}$, we will be interested in solutions of the integro-differential equation
\begin{align}
\left\{\begin{array}{cclrl}
\I u(x)&=&0 &\text{if}  &x\in \Omega,\\
u(x)&=&f(x) &\text{if} &x\in \R^N\setminus \Omega.
\end{array}\right.\label{PDE}
\end{align}
(This is the Dirichlet problem for the infinity fractional Laplacian. In Section \ref{obstacle} we will also consider a double obstacle problem associated to the infinity fractional Laplacian.)

As we will see, a natural space for the data $f$ is the set of uniformly H\"older continuous with exponent $2s-1$, that is,
\begin{align}
\sup_{x,y\in \R^N\setminus \Omega}\frac{|f(x)-f(y)|}{|x-y|^{2s-1}}<\infty.\notag
\end{align}
If $f$ belongs to this space, we will show that (sub-super)solutions of (\ref{PDE}), if they exist, are also uniformly H\"older continuous with exponent $2s-1$,
and have a H\"older constant less then or equal to the constant for $f$. 
This is analogous to the well known absolutely minimizing property of the infinity Laplacian \cite{MR1861094,MR1218686}, and is argued through  ``comparison with cusps'' of the form
\begin{align}
\mathcal{C}(x)=A|x-x_0|^{2s-1}+B.\label{cuspdfn}
\end{align}
Indeed, this follows from the fact that these cusps satisfy $\I\mathcal{C}(x)=0$ at any point $x\neq x_0$  (see Lemma \ref{lemma:fund solution}),
so they can be used as barriers.
If the data $f$ in (\ref{PDE}) is assumed uniformly Lipschitz and bounded, one can use this H\"older continuity property to get enough compactness and regularity
to show that, as $s\rightarrow 1$, solutions converge uniformly to the (unique) solution of the Dirichlet problem for the infinity Laplacian $\triangle_\infty u =D^2u[\nabla u,\nabla u]$.

Let us point out that uniqueness of viscosity solutions to \eqref{PDE} is not complicated, see for instance Theorem \ref{comparison}.
Instead, the main obstacle here is in the existence theory, and the problem comes from the discontinuous behavior of $\I\phi$ at points where $\nabla\phi=0$.
Consider for example a function obtained by taking the positive part of a paraboloid: $p:\mathbb{R}^2\rightarrow\mathbb{R}$ is given by $p(x,y)=(1-2x^2-y^2)\vee 0$.
(Here and in the sequel,
$\vee$ [resp. $\wedge$] denotes the maximum [resp. minimum] of two values.)  Note
\begin{align}
\lim_{h\rightarrow 0} \I p(h,0)\neq \lim_{h\rightarrow 0}\I p(0,h)\neq \I p(0,0).\notag
\end{align}
In particular $\I\phi(x)$ can be discontinuous even if $\phi$ is a very nice function, and in particular $\I$ is unstable under uniform limit at points where the
limit function has zero derivative.
This is actually also a feature of the infinity Laplacian if one defines $\Delta_\infty \phi=\frac{D^2\phi[\nabla \phi,\nabla \phi]}{|\nabla\phi|^2}$.
However, in the classical case, this problem is ``solved'' since one actually considers the operator $D^2\phi[\nabla \phi,\nabla \phi]$, so with the latter definition the infinity Laplacian
is zero when $\nabla\phi=0$. In our case we cannot adopt this other point of view, since $u=0$ in $\Omega$ would always be a solution whenever
$f:\R^N\setminus \Omega\to \R$ vanishes on $\partial\Omega$, even if $f$ is not identically zero.
As we will discuss later, in game play this instability phenomenon of the operator is expressed in unintuitive strategies which stem
from the competition of local and non-local aspects of the operator, see Remark \ref{gameplayex}.

In order to prevent such pathologies and avoid this (analytical) problem, we will restrict ourselves to 
situations where we are able to show that $\nabla u\neq 0$ (in the viscosity sense) so that $\I$ will be stable.
This is a natural restriction guaranteeing the players will always point in opposite directions.  Using standard techniques we also show uniqueness of solutions on compact sets,
and uniqueness on non-compact sets in situations where the operator is stable.

We will consider two different problems: (D) the Dirichlet problem; (O) a double obstacle problem.
As we will describe in the next section, they both have a natural interpretation as the limit of value functions for a ``non-local tug-of-war''.
Under suitable assumptions on the data, we can establish ``strict uniform monotonicity'' (see Definition \ref{unifmonotonedefn}) of the function constructed using Perron's method (which at the beginning we do not know to be a solution),
so that we can prove prove existence and uniqueness of solutions.

In situation (D) we consider $\Omega$ to be an infinite strip with data $0$ on one side and $1$ on the other,
this is the problem given by (\ref{PDEm}).  Assuming some uniform regularity on the boundary of $\Omega$ we construct suitable barriers
which give estimates on the growth and decay of the ``solution'' near $\partial\Omega$ implying strict uniform monotonicity.

In situation (O) we consider two obstacles, one converging to $0$ at negative infinity along the $e_1$ axis which the solution must lie below,
and one converging to $1$ at plus infinity along the $e_1$ axis which the solution must lie above. Then, we will look for a function $u:\R^n \to [0,1]$
which solves the infinity fractional Laplacian whenever it does not touch one of the obstacles. This is the problem given by \eqref{obstaclePDE}.
We prove that the solution must coincide with the obstacles near plus and minus infinity, and we use this to deduce strict uniform monotonicity.
In addition we demonstrate Lipschitz regularity for solutions of the obstacle problem, and analyze how the solution approaches the obstacle.

Before concluding this section we would like to point out a related work, \cite{CLM}, which considers H\"older extensions from a variational point of view as opposed to the game theoretic approach we take here.  In \cite{CLM} the authors construct extensions by finding minimizers for the $W^{ps,\infty}(\Omega)$ norm with $s\in (0,1)$, then taking the limit as $p\rightarrow \infty$.

Now we briefly outline the paper:  In Section \ref{derivation} we give a detailed (formal) derivation for the operator $\Delta^s_\infty$,
and we introduce the concept of viscosity solutions for this operator.
In Section \ref{comparisonsec} we prove a comparison principle on compact sets and demonstrate H\"older regularity of solutions.
This section also contains a stability theorem and an improved regularity theorem.
 In Section \ref{monotone} we investigate a Dirichlet monotone problem, and in Section \ref{obstacle} we investigate a monotone double obstacle problem.

\section{Derivation of the Operator and Viscosity Definitions}\label{derivation}
\subsection{Heuristic Derivation}We give a heuristic derivation of $\I$ by considering two different non-local versions of the two player tug-of-war game:
\begin{itemize}
\item[(D)]  Let $\Omega$ be an open,
simply connected subset of $\mathbb{R}^N$ where the game takes place, and let $f:\mathbb{R}^N\rightarrow \mathbb{R}$
describe the payoff function for the game (we can assume $f$ to be defined on the whole $\R^N$).
The goal is for player one to maximize the payoff while player two attempts to minimize the payoff. 
The game is started at any point $x_0\in \Omega$, and let $x_n$ denote the position of the game at the the beginning of the $n$th turn.
Then both players pick a direction vector, say $y_n,z_n\in S^{N-1}$ (here and in the sequel, $S^{N-1}$ denotes the unit sphere), and the two players 
observe a stochastic process $X_t$ on the real line starting from the origin.  This process is embedded into the game by the function $g_n:\mathbb{R}\rightarrow \mathbb{R}^N$ defined:
\begin{align}
g_n(X_t)=\left\{
\begin{array}{lcl}
X_t v_n+x_{n} &\text{if} &X_t\geq 0,\\
X_t w_n+x_{n} &\text{if} &X_t<0.
\end{array}\right. \notag
\end{align}
On each turn the stochastic process is observed for some predetermined time $\epsilon>0$ which is the same for all turns.
If the image of the stochastic process remains in $\Omega$ the position of the game moves to $x_{n+1}=g_n(X_{\epsilon})$, and the game continues.
If the image of the stochastic process leaves $\Omega$, that is, $g_n(X_\epsilon)\in \mathbb{R}^N\setminus\Omega$, then the game is stopped and the payoff of the game is $f(g_n(X_\epsilon))$. 

\item[(O)] Consider two payoff functions, $\Gamma^+,\Gamma^-:\mathbb{R}^N\rightarrow\mathbb{R}$, such that $\Gamma^+(x)\geq\Gamma^-(x)$ for all $x\in\mathbb{R}^N$.
Again the goal is for player one to maximize the payoff while player two attempts to minimize the payoff.
The game starts at any point $x_0\in\mathbb{R}^N$ (now there is no boundary data), and let $x_n$ denote the position of the game at the the beginning of the $n$th turn.
In this case, at the beginning of each turn, both players are given the option to stop the game: if player one stops the game the payoff is $\Gamma^-(x_n)$,
and if player two stops the game the payoff is $\Gamma^+(x_n)$.  If neither player decides to stop the game, then both players pick a direction vector and observe a stochastic process as in (D).
Then, the game is moved to the point $x_{n+1}=g_n(X_\epsilon)$ as described in case (D), and they continue playing in this way until one of the players decides to stop the game.
\end{itemize}
The game (D) will correspond to the Dirichlet problem \eqref{PDEm}, while (O) corresponds to the double obstacle problem \eqref{obstaclePDE} where the payoff functions $\Gamma^+,\Gamma^-$
will act as an upper and lower obstacles.

In order to derive a partial differential equations associated to these games, we use the dynamic programming principle
to write an integral equation whose solution represents the expected value of the game starting at $x\in \Omega$.  Denote by $p^t(\eta)$
the transition density of the stochastic process observed at time $t$, so that for a function $f:\mathbb{R}\rightarrow \mathbb{R}$ the expected value of $f(X_t)$ is $p^t\ast f(0)$.  The expected value $u(x)$ of the game starting at $x\in \Omega$, and played with observation time $\epsilon$, satisfies
\begin{align}
2u(x)= \sup_{y\in S^{N-1}}\inf_{z\in S^{N-1}}\left\{ \int_0^\infty p^\epsilon(\eta)u(x+\eta y)\,d\eta
+ \int_0^\infty p^\epsilon(\eta)u(x-\eta z)\,d\eta\right\},\notag
\end{align}
or equivalently
\begin{align}
0= \sup_{y\in S^{N-1}}\inf_{z\in S^{N-1}}\left\{\int_0^\infty \frac{p^\epsilon(\eta)}{\epsilon}[u(x+\eta y)+u(x-\eta z)-2u(x)]\,d\eta\right\}.\label{approxoperator}
\end{align}
We are interested in the limit $\epsilon\rightarrow 0$.

We now limit the discussion to the specific case where the stochastic process is a one dimensional symmetric $s$-stable Levy process for $s\in (\frac{1}{2},1)$.  That is
\begin{align}
\mathbb{E}_x\left[e^{i\xi \cdot (X_t-X_0)}\right]=e^{-t|\xi|^{2s}}.\notag
\end{align}
It is well known that this process has an infinitesimal generator
\begin{align}
-(-\triangle)^su(x)=2(1-s)\int_{0}^\infty\frac{u(x+\eta )+u(x-\eta)-2u(x)}{\eta^{1+2s}}\notag
\end{align}
and a transition density which satisfies
\begin{align}
p^\epsilon(\eta)\sim \frac{\epsilon}{(\epsilon^\frac{1}{s}+\eta^2)^{\frac{2s+1}{2}}}.\notag
\end{align}
Hence, in the limit as $\epsilon \to 0$ (\ref{approxoperator}) becomes
\begin{align}
0= \sup_{y\in S^{N-1}}\inf_{z\in S^{N-1}}\left\{\int_0^\infty \frac{u(x+\eta y)+u(x-\eta z)-2u(x)}{\eta^{1+2s}}\,d\eta\right\}.\label{operator}
\end{align}
As we will show below, it is not difficult to check that, if $u$ is smooth, this operator coincides with the one in Definition \ref{IFLdefn}.
From this game interpretation, we can also gain insight into the instability phenomenon mentioned in the introduction:

\begin{remark}\label{gameplayex}{\rm 
As evident from Definition \ref{IFLdefn} the players choice of direction at each term is weighted heavily by the gradient, a local quantity
(and in the limit $\epsilon=0$, it is uniquely determined from it).
However, after this direction is chosen, the jump is done accordingly to the stochastic process, a non-local quantity, and it may happen
that the choice dictated by the gradient is exactly the opposite to what the player would have chosen in order to maximize its payoff.

Consider for example the following situation: $\Omega\subset\mathbb{R}^2$ is the unit ball
centered at the origin, and $f(x,y)$ is a smooth non-negative function (not identically zero) satisfying $f(x,y)=f(x,-y)$ and supported in the unit ball centered at $(2,0)$.
If a solution $u$ for (\ref{PDE}) exists, by uniqueness (Theorem \ref{comparison}) $u$ is symmetric with respect to $y=0$. Hence it must attain a local maximum at a point $(r_0,0)$ with $r_0\in (-1,1)$,
and there are points $(r,0)$ with $r>r_0$ such that the gradient (if it exists) will have direction $(-1,0)$.
Starting from this point the player trying to maximize the payoff will pick the direction $(-1,0)$ since this will be the direction of the gradient,
but the maximum of $f$ occurs exactly in the opposite direction $(1,0)$.}
\end{remark}

\subsection{Viscosity Solutions}
As we said, the operator defined in \eqref{operator} coincides with the one in Definition \ref{IFLdefn} when $u$ is smooth.
If $u$ is less regular, one can make sense of the $\inf\sup$ with a ``viscosity solution'' philosophy (see \cite{MR1118699}), but first let us define $C^{1,1}$ functions at a point $x_0$:
\begin{definition}\label{C11defn}
 A function $\phi$ is said to be $C^{1,1}(x_0)$, or equivalently ``$C^{1,1}$ at the point $x_0$'' if there is a vector $p\in \mathbb{R}^N$ and numbers $M,\eta_0>0$ such that
\begin{align}
\label{eq:C11}
 |\phi(x_0+x)-\phi(x_0)-p\cdot x|\leq M |x|^2
\end{align}
for $|x|<\eta_0$.  We define $\nabla \phi(x_0):=p$.  
\end{definition}
It is not difficult to check that the above definition of $\nabla \phi(x_0)$ makes sense, that is, if $u$ belongs to $C^{1,1}(x_0)$ then there exists a unique vector $p$
for which \eqref{eq:C11} holds.

Turning back to (\ref{operator}), let $u\in C^{1,1}(x)$ be bounded and H\"older continuous.  We can say $u$ is a supersolution at $x\in \Omega$ if, for any $\epsilon>0$, there exists
 $z_\epsilon\in S^{N-1}$ such that
\begin{align}
\sup_{y\in S^{N-1}}\left\{\int_0^\infty \frac{u(x+\eta y)+u(x-\eta z_\epsilon)-2u(x)}{\eta^{1+2s}}\,d\eta\right\}\leq \epsilon.\label{supop}
\end{align}
If $\nabla u(x)=0$ the above integral is finite for any choice of $y$ and $z_\epsilon$. As $S^{N-1}$ is compact, there is
a subsequence $\epsilon\rightarrow 0$ and $z_\epsilon\rightarrow z_0$ such that, in the limit,
\begin{align}
\sup_{y\in S^{N-1}}\left\{\int_0^\infty \frac{u(x+\eta y)+u(x-\eta z_0)-2u(x)}{\eta^{1+2s}}\,d\eta\right\}\leq 0.\notag
\end{align}
So, in the case $\nabla u(x)=0$, we say that $u$ is a supersolution if there is a $z_0$ such that the above inequality holds.

If $\nabla u(x)\neq 0$ we rewrite
\begin{align}
\int_0^\infty &\frac{u(x+\eta y)+u(x-\eta z_\epsilon)-2u(x)}{\eta^{1+2s}}\,d\eta\notag\\
&\ \ \ \ =\int_0^\infty \frac{u(x+\eta y)+u(x-\eta z_\epsilon)-\eta\nabla u(x_0)\cdot(y-z_\epsilon)-2u(x)}{\eta^{1+2s}}\,d\eta \notag\\
&\ \ \ \ \ \ \ \ \ \ +\nabla u(x_0)\cdot(y-z_\epsilon)\int_0^\infty \eta^{-2s}\,d\eta.\notag
\end{align}
The first integral on the right hand side is convergent for all choices of $y$, $z_\epsilon$ but the second diverges when $s>\frac{1}{2}$. (Strictly speaking, one should argue
in the limit of (\ref{approxoperator}) to understand the second integral.)  Let $v\in S^{N-1}$ denote the direction of $\nabla u(x)$.  Since $v$ is a possible choice
for $y$ and $\nabla u(x)\cdot(v-z_\epsilon)\geq 0$ for any choice of $z_\epsilon$ we are compelled to choose $z_\epsilon=v$.  Likewise, once we set
$z_\epsilon=v$ the supremum in (\ref{supop}) is obtained when $y=v$. Hence, in the case $\nabla u(x)\neq0$, we say that $u$ is a supersolution if 
\begin{align}
\int_0^\infty \frac{u(x+\eta v)+u(x-\eta v)-2u(x)}{\eta^{1+2s}}\,d\eta\leq 0, \qquad v=\frac{\nabla u(x)}{|\nabla u(x)|}\notag
\end{align}
A similar argument can be made for subsolutions and with these considerations the right hand side of (\ref{operator}) leads to Definition \ref{IFLdefn}.

In addition to $u\in C^{1,1}$ we have assumed $u$ is bounded and H\"older continuous.
As we will demonstrate in Section \ref{comparison}, both assumptions can be deduced from the data when the payoff function $f$ is bounded and uniformly H\"older continuous.
When $u\notin C^{1,1}$ we use the standard idea of test functions for viscosity solutions, replacing $u$ locally with a $C^{1,1}$ function which touches it from below or above.  

\begin{definition}\label{subsupdefn}
An upper [resp. lower] semi continuous function $u:\mathbb{R}^N\rightarrow\mathbb{R}$ is said to be a \emph{subsolution} [resp. \emph{supersolution}] at $x_0$, and we write $\I u(x_0)\geq 0$ [resp. $\I u(x_0)\leq 0$], if every time all of the following happen:
\begin{itemize}
\item $B_r(x_0)$ is an open ball of radius $r$ centered at $x_0$,
\item $\phi\in C^{1,1}(x_0)\cap C(\bar{B}_r(x_0))$,
\item $\phi(x_0)=u(x_0)$,
\item $\phi(x)>u(x)$ [resp. $\phi(x)<u(x)$] for every $x\in B_r(x_0)\setminus \{x_0\}$,
\end{itemize}
we have $\I\tilde{u}(x_0)\geq 0$ [resp. $\I\tilde{u}(x_0)\leq 0$], where
\begin{align}
\tilde{u}(x):=\left\{
\begin{array}{ccl}
\phi(x) &\text{if}  &x\in  B_r(x_0)\\
u(x) &\text{if} &x\in \mathbb{R}^N\setminus B_r(x_0).
\end{array}\right.\label{subsupcut}
\end{align}
\end{definition}
In the above definition we say the test function $\phi$ ``touches $u$ from above [resp. below] at $x_0$''.
We say that $u:\Omega\rightarrow\mathbb{R}$ is a \textit{subsolution} [resp. \textit{supersolution}]\
if it is a subsolution [resp. supersolution] at every point inside $\Omega$.
We will also say that a function $u:\Omega\rightarrow\mathbb{R}$ is a ``\textit{subsolution} [resp. \textit{supersolution}] \textit{at non-zero gradient points}''
if it satisfies the subsolution [resp. supersolution] condition on Definition \ref{subsupdefn} only when $\nabla\phi(x)\neq 0$.
If a function is both a subsolution and a supersolution, we say it is a \textit{solution}. 

With these definitions, game (D) leads to the Dirichlet problem (\ref{PDEm}), and game (O) leads to the double obstacle problem (\ref{obstaclePDE}).

\subsection{Further Considerations}\label{otherdefn}
As mentioned in the introduction, the definition of solution we adopt is not stable under uniform limits on compact sets.  We observe here a weaker definition for solutions which is stable under such limits by modifying Definition \ref{subsupdefn} when the test function has a zero derivative.  

One may instead define a subsolution to require only that
\begin{align}
\sup_{y\in S^{N-1}}\int_0^\infty\frac{\tilde{u}(x+\eta y)+\tilde{u}(x-\eta y)-2\tilde{u}(x)}{\eta^{1+2s}}\,d\eta\geq 0 \label{altsub}
\end{align}
when $u$ is touched from above at a point $x_0$ by a test function $\phi$ satisfying $\nabla\phi(x_0)=0$.
Likewise, one may choose the definition for a supersolution to require only
\begin{align}
\inf_{y\in S^{N-1}}\int_0^\infty\frac{\tilde{u}(x+\eta y)+\tilde{u}(x-\eta y)-2\tilde{u}(x)}{\eta^{1+2s}}\,d\eta\leq 0 \label{altsup}
\end{align}
when $u$ is touched from below at a point $x_0$ by a test function $\phi$ satisfying $\nabla\phi(x_0)=0$.
When $\phi$ satisfies $\nabla\phi(x_0)\neq 0$ we refer back to Definition \ref{subsupdefn}.
It is straightforward to show this definition is stable under uniform limits on compact sets,
and it should not be difficult to prove that
the expected value functions coming from the non-local tug-of-war game
converge to a viscosity solution of \eqref{altsub}-\eqref{altsup}.
However, although this definition may look natural,
unfortunately solutions are not unique and no comparison principle holds.
(Observe that, in the local case, by Jensen's Theorem \cite{MR1218686} one does not need to impose any condition 
at points with zero gradient.)

For example, in one dimension consider the classical fractional Laplacian on $(0,1)$, with non-negative data in
the complement vanishing both at $0$ and $1$.  This has a unique non-trivial positive solution $v(x)$.
 When extended in one more dimension as $u(x,y)=v(x)$ we recover
 a non-trivial solution in the sense of Definition \ref{subsupdefn}
to the problem in the strip $(0,1)\times \mathbb{R}$ with corresponding boundary data, and of course this is also 
a solution in the sense of \eqref{altsub}-\eqref{altsup}. However, with the latter defintion,
also $u\equiv 0$ is a solution.  

To further show the lack of uniqueness for \eqref{altsub}-\eqref{altsup},
in the appendix we also provide an example of a simple bi-dimensional geometry for which
the boundary data is positive and compactly supported, and $u\equiv 0$ is a solution.
However, for this geometry there is also a positive subsolution in the sense of \eqref{altsub}-\eqref{altsup},
which show the failure of any comparison principle.

In light of these examples, we decided to adopt Definition \ref{subsupdefn}, for which a comparison principle holds
(see Theorems \ref{comparison} and \ref{monocompthrm}).

\section{Comparison Principle and Regularity}\label{comparisonsec}

\subsection{Comparison of Solutions}
Our first goal is to establish a comparison principle for subsolutions and supersolutions.
To establish the comparison principle in the ``$\nabla u(x_0)=0$'' case we will make use of the following maximal-type lemma:
\begin{lemma}\label{maxlemma}
Let $u,w\in C^{1,1}(x_0)\cap BC(\mathbb{R}^N)$ be such that $\nabla u(x_0)=\nabla w(x_0)=0$.  Then
\begin{align}
2\inf_{z\in S^{N-1}}\int_0^\infty&\frac{[u-w](x_0-\eta z)-[u-w](x_0)}{|\eta|^{1+2s}}\,d\eta\notag\\
&\leq \I u(x_0)-\I w(x_0) \notag\\
&\leq 2\sup_{y\in S^{N-1}}\int_0^\infty\frac{[u-w](x_0+\eta y)-[u-w](x_0)}{|\eta|^{1+2s}}\,d\eta.\notag
\end{align}
\end{lemma}
\begin{proof}
We use the notation
\begin{align}
L(u,y,x_0)=\int_0^\infty\frac{u(x_0+\eta y)-u(x_0)}{|\eta|^{1+2s}}\,d\eta\notag
\end{align}
so that
\begin{align}
\I u(x_0)=\sup_{y\in S^{N-1}}L(u,y,x_0)+\inf_{y\in S^{N-1}}L(u,y,x_0).\notag
\end{align}

For any $\delta >0$ there exists $\hat{y},\bar{y}\in S^{N-1}$ such that 
\begin{align}
\sup_{y\in S^{N-1}}L(u,y,x_0) - L(u,\hat{y},x_0)< \delta,\ \ \ \sup_{y\in S^{N-1}}L(w,y,x_0)-L(w,\bar{y},x_0) <\delta.\notag
\end{align}
This implies
\begin{align}
L(u,\bar{y},x_0) \leq &\sup_{y\in S^{N-1}}L(u,y,x_0) < \delta +L(u,\hat{y},x_0),\notag\\
-L(w,\bar{y},x_0)-\delta< &-\sup_{y\in S^{N-1}}L(w,y,x_0)\leq -L(w,\hat{y},x_0).\notag
\end{align}
All together, and using the linearity of $L$,
\begin{align}
\inf_{z\in S^{N-1}}&L(u-w,-z,x_0)-\delta\notag\\ 
&< \sup_{y\in S^{N-1}}L(u,y,x_0)-\sup_{y\in S^{N-1}}L(w,y,x_0)\notag\\
&\ \ \ \ \ \ \ \ <\sup_{y\in S^{N-1}}L(u-w,y,x_0)+\delta.\notag
\end{align}
A similar argument holds for the infimums in the definition of $\I$ and the proof is completed by combining them and letting $\delta\rightarrow 0$.
\end{proof}

We now show a comparison principle on compact sets. We assume that our functions
grow less than $|x|^{2s}$ at infinity (i.e., there exist $\alpha <2s$, $C>0$, such that $|u(x)| \leq C(1+|x|)^\alpha$), so that the integral defining $\I$ is convergent at infinity.
Let us remark that, in all the cases we study, the functions will always grow at infinity at most as $|x|^{2s-1}$,
so this assumption will always be satisfied.

\begin{theorem}\emph{(Comparison Principle on Compact Sets)}\label{comparison}
Assume $\Omega$ is a bounded open set.  Let $u,w:\mathbb{R}^N\rightarrow \mathbb{R}$, be two continuous functions such that
\begin{itemize}
\item $\I u(x)\geq 0$ and $\I w(x)\leq 0$ for all $x\in\Omega$ (in the sense of Definition \ref{subsupdefn}),
\item $u$, $w$ grow less than $|x|^{2s}$ at infinity,
\item $u\leq w$ in $\mathbb{R}^N\setminus \Omega$.
\end{itemize}
Then $u\leq w$ in $\Omega$.
\end{theorem}

The strategy of the proof is classical: assuming by contradiction there is a point $x_0\in \Omega$ such that $u(x_0)>w(x_0)$,
we can lift $w$ above $u$, and then lower it until it touches $u$ at some point $\bar x$. Thanks to the assumptions, (the lifting of) $w$ will be strictly greater
then $u$ outside of $\Omega$, and then we use the sub and supersolution conditions at $\bar x$ to find a contradiction.
However, since in the viscosity definition we need to touch $u$ [resp. $w$] from above [resp. below] by a $C^{1,1}$ function at $\bar x$,
we need first to use sup and inf convolutions to replace $u$ [resp. $w$] by a semiconvex [resp. semiconcave] function:
\begin{definition}
Given a continuous function $u$, the ``sup-convolution approximation'' $u^\epsilon$ is given by
\begin{align}
u^\epsilon(x_0)=\sup_{x\in \mathbb{R}^N} \left\{u(x)+\epsilon-\frac{|x-x_0|^2}{\epsilon}\right\}.\notag
\end{align}
Given a continuous function $w$, the ``inf-convolution approximation'' $w_\epsilon$ is given by
\begin{align}
w_\epsilon(x_0)=\inf_{x\in \mathbb{R}^N}\left\{ w(x)-\epsilon+\frac{|x-x_0|^2}{\epsilon}\right\}.\notag
\end{align}
\end{definition}

We state the following lemma without proof, as it is standard in the theory of viscosity solutions (see, for instance, \cite[Section 5.1]{MR1351007}).
\begin{lemma}\label{infsupconvlemma}
Assume that $u,w:\mathbb{R}^N\rightarrow \mathbb{R}$ are two continuous functions which grow at most as $|x|^{2}$ at infinity.
The following properties hold:
\begin{itemize}
\item $u^\epsilon \downarrow u$ [resp. $w_\epsilon \uparrow w$] uniformly on compact sets as $\epsilon\rightarrow 0$.
Moreover, if $u$ [resp. $w$] is uniformly continuous on the whole $\R^N$, then the convergence is uniform on $\R^N$.
\item At every point there is a concave [resp. convex] paraboloid of opening $2/\epsilon$
touching $u^\epsilon$ [resp. $w_\epsilon$] from below [resp. from above]. (We informally refer to this property
by saying that $u^\epsilon$ [resp. $w_\epsilon$] is $C^{1,1}$ from below [resp. above].)
\item If $\I u\geq 0$ [resp. $\I w\leq 0$] in $\Omega$ in the viscosity sense, then $\I u^\epsilon \geq 0$ [resp. $\I w^\epsilon \leq 0$] in $\Omega$.
\end{itemize}
\end{lemma}

\begin{proof}[Proof of Theorem \ref{comparison}]
Assume by contradiction that there is a point $x_0\in \Omega$ such that $u(x_0)>w(x_0)$. Replacing $u$ and $w$ by $u^\epsilon$ and $w_\epsilon$, 
we have  $u^\epsilon(x_0)-w_\epsilon(x_0) \geq c>0$ for $\epsilon$ sufficiently small,
and $(u^\epsilon-w_\epsilon)\vee 0 \to 0$ as $\epsilon \to 0$ locally uniformly outside $\Omega$.
Thanks to these properties, the continuous function $u^\epsilon-w_\epsilon$ attains its maximum over $\overline\Omega$ at some interior point $\bar x$ inside $\Omega$.
Set $\delta = u^\epsilon(\bar x)-w_\epsilon(\bar x)\geq c>0$.
Since $u^\epsilon$ is $C^{1,1}$ from below, $w_\epsilon+\delta$ is $C^{1,1}$ from above, and $w_\epsilon+\delta$ touches $u^\epsilon$ from below at $\bar x$,
it is easily seen that both $u^\epsilon$ and $w_\epsilon+\delta$ are $C^{1,1}$ at $\bar x$ (in the sense of Definition \ref{C11defn}), and that
$\nabla u^{\epsilon}(x_0)=\nabla w_\epsilon(x_0)$. So, can evaluate $\I u^\epsilon(x_0)$ and $\I w_\epsilon(x_0)$ directly
without appealing to test functions.

As $u^\epsilon$ is a subsolution and $w_\epsilon$ is a supersolution, we have $\I u^\epsilon(\bar x)\geq 0$ and $\I(w_\epsilon+\delta)(\bar x)=\I w_\epsilon(\bar x)\leq 0$.
We now break the proof into two cases, reflecting the definition of $\I$.\\
{\bf Case I:} Assume $\nabla u^\epsilon(\bar x)=\nabla w_\epsilon(\bar x)\neq 0$ and let $v\in S^{N-1}$ denote the direction of that vector.  Then,
\begin{align}
0&\geq \I(w_\epsilon+\delta)(\bar x)-\I(u^\epsilon)(\bar x)\notag\\
&\geq \int_{0}^\infty\frac{[w_\epsilon-u^\epsilon](\bar x+\eta v)+[w_\epsilon-u^\epsilon](\bar x-\eta v)-2[w_\epsilon-u^\epsilon](\bar x)}{\eta^{1+2s}}\,d\eta.\notag
\end{align}
As $[w_\epsilon+\delta-u^\epsilon](\bar x)=0$ and $[w_\epsilon+\delta-u^\epsilon](x)\geq 0$ on $\mathbb{R}^N$, we conclude $[w_\epsilon+\delta-u^\epsilon]= 0$ along
the line $\{\bar x+\eta v\}_{\eta\in\R}$, which for $\epsilon$ small contradicts the assumption $u\leq w$ in $\mathbb{R}^N\setminus \Omega$.\\
{\bf Case II:} Assume $\nabla u^\epsilon(\bar x)=\nabla w_\epsilon(\bar x)= 0$.
We proceed similarly to the previous case but use in addition Lemma \ref{maxlemma} applied to $u^\epsilon$ and $w_\epsilon$.  Then,
\begin{align}
0&\geq \I(\tilde{w}_\epsilon+\delta)(\bar x)-\I(\tilde{u}^\epsilon)(\bar x)\notag\\
&\geq 2\inf_{z\in S^{N-1}}\left(\int_{0}^\infty\frac{[w_\epsilon-u^\epsilon](\bar x-\eta z)-[w_\epsilon-u^\epsilon](\bar x)}{\eta^{1+2s}}\,d\eta\right).\notag
\end{align}
Using $[w_\epsilon+\delta-u^\epsilon](\bar x)=0$ and $[w_\epsilon+\delta-u^\epsilon]\geq 0$ on $\mathbb{R}^N$ we deduce the existence of a ray $\{\bar x+\eta v\}_{\eta\geq 0}$,
such that $[w_\epsilon+\delta-u^\epsilon]= 0$ along this ray. This contradicts again the assumption $u\leq w$ in $\mathbb{R}^N\setminus \Omega$, and our proof is complete.
\end{proof}

\subsection{Regularity of Solutions}
The aim of this section is to show that (sub/super)solutions are H\"older continuous of exponent $2s-1$ whenever the boundary data is.
\begin{definition}\label{def:holder}
A function $f:\Omega\to\R$ is H\"older continuous with exponent $\gamma$ on the set $\Omega$ if
\begin{align}
[f]_{C^{0,\gamma}(\Omega)}:=\sup_{x,y \in\Omega}\frac{|f(x)-f(y)|}{|x-y|^\gamma}<\infty.\notag
\end{align}
In this case we write $f\in C^{0,\gamma}(\Omega)$. Moreover, if
$$
\lim_{\delta \to 0} \sup_{x,y \in\Omega,\, |x-y|\leq \delta}\frac{|f(x)-f(y)|}{|x-y|^\gamma}=0,
$$
then we write $f\in C_0^{0,\gamma}(\Omega)$.
\end{definition}
Let us observe that $C_0^{0,\gamma}(\Omega)$ is a separable subspace of the (non-separable) Banach space $C^{0,\gamma}(\Omega)$.
This subspace can be characterized as the closure of smooth function with respect to the norm
$\|f\|_{C^{0,\gamma}(\Omega)}=\|f\|_{L^\infty(\Omega)}+[f]_{C^{0,\gamma}(\Omega)}$. In this section we will show that solutions belong to $C^{0,\gamma}$.
Then, in Section \ref{sect:improved regularity} we will prove a Liouville-type theorem which shows that solutions are actually $C_0^{0,\gamma}$.

As mentioned in the introduction, the natural space for solutions of this problem is $C^{0,2s-1}(\Omega)$. This follows from the fact that the cusps
\begin{align}
\mathcal{C}(x)=A|x-x_0|^{2s-1}+B\notag
\end{align}
are solutions of $\I$ at all points but the tip of the cusp, so they can be used as barriers.

\begin{lemma}\label{lemma:fund solution}
For all $x\neq x_0$, $\I\mathcal{C}(x)=0$.
\end{lemma}
\begin{proof}
For any $x\neq x_0$, $\nabla \mathcal{C}(x)\neq 0$ and the direction of $\nabla\mathcal{C}(x)$ is the same as the direction of $x-x_0$.
Moreover $\mathcal{C}$ is smooth in a neighborhood of $x$ so we may evaluate $\I\mathcal{C}(x)$ classically.
Thus, to evaluate $\I\mathcal{C}(x)$ we need only to evaluate the 1-D $s$-fractional Laplacian on the 1-D function $\R\ni\eta \mapsto |\eta|^{2s-1}$.
It can be checked through Fourier transform that $|\eta|^{2s-1}$ is the fundamental solution for the 1-D $s$-fractional Laplacian (i.e.,
$\I |\eta|^{2s-1} = c_s \delta_0$), which completes the proof.
\end{proof}

By considering data $f\in C^{0,2s-1}(\R^N\setminus \Omega)$ and comparing (sub/super)solutions of (\ref{PDE}) with these cusps, we will establish regularity.
The intuition is that if $u$ is for instance a subsolution, and we choose a cusp with $A>0$ sufficiently large, for every $x_0 \in \Omega$ we can first raise the cusp 
centered at $x_0$ above the solution, and then lower the cusp until it touches the function $u$ at some point.
If we have assumptions on the data that ensure that the contact point cannot occur outside $\Omega$ (at least for $A$ sufficiently large),
then by the (strict) comparison principle
this contact point can only happen at the tip, from which the H\"older regularity follows easily (see Figure \ref{fig1}).

\begin{figure}
\centerline{\epsfysize=1truein\epsfbox{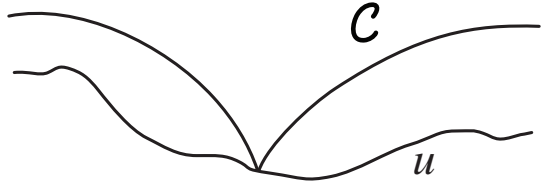}}
\caption{The cusp $\mathcal{C}$ touching $u$ from above.}
\label{fig1}
\end{figure}

We state the following result for subsolutions. By changing $u$ with $-u$, it also holds for supersolutions (with obvious modifications).

\begin{theorem}\label{holdcontthrm}
Let $u$ be a subsolution of $\I$ in the sense of Definition \ref{subsupdefn} inside an open set $\Omega$, with $u=f$ on $\R^N\setminus \Omega$.
Furthermore, assume there exist $A>0$, $C_0<\infty$, such that
\begin{align}
\label{eq:control f}
\sup_{z\in \R^N\setminus \Omega}\left\{f(z)-A|x-z|^{2s-1}\right\}\leq u(x) \leq C_0\qquad \forall\,x \in \Omega.
\end{align}
Then $u\in C^{0,2s-1}(\R^N)$ and
$$
[u]_{C^{0,2s-1}(\R^N)} \leq A.
$$
In particular, since $u=f$ on $\R^N\setminus \Omega$,
$$
|u(x)| \leq \inf_{z\in \R^N\setminus\Omega}\left\{f(z)+ A|x-z|^{2s-1}\right\}\qquad \forall\,x \in \Omega.
$$
\end{theorem}
\begin{proof}
As we said above, this theorem is proved through comparison with cusps. 

Fix $x_0\in \Omega$. We will show that, for any $x\in \mathbb{R}^N$,
\begin{align}
\label{eq:Holder}
u(x)\leq u(x_0)+A|x-x_0|^{2s-1}.
\end{align}
Choosing $x \in \Omega$ and exchanging the role of $x$ and $x_0$, this will prove the result.

Fix $\epsilon>0$ and consider cusps of the form $\mathcal{C}(x)=B+(A+\epsilon)|x-x_0|^{2s-1}$. Thanks to assumption \eqref{eq:control f}
we see that, if we choose $B=u(x_0)$, then $\mathcal C>f$ on $\R^N\setminus \Omega$. This implies that, if we first choose $B=C_0$ with $C_0$ as in \eqref{eq:control f}
so that $\mathcal C \geq u$ on the whole $\R^n$, and then we lower $B$ until $\mathcal C$ touches $u$ from above at a point $\bar x$, then $\bar x \in\Omega$.

We claim that $\bar x=x_0$. Indeed, if not, $\mathcal C$ would be smooth in a neighborhood of $\bar x$,
and we can use it as a test function to construct $\tilde{u}$ as in (\ref{subsupcut}). (Observe that, strictly speaking, we do not necessarily have $\mathcal{C}(x)>u(x)$
when $x\neq \bar{x}$, but this can be easily fixed by an easy approximation argument.)
Since $u$ is a subsolution it must be that $\I\tilde{u}(\bar{x})\geq 0$. So, for any $r\in(0,|\bar x-x_0|)$ we have
\begin{align}
0\leq \I\tilde{u}(\bar{x})&=
\int_{0}^r\frac{\mathcal{C}(\bar{x}+\eta v)+\mathcal{C}(\bar{x}-\eta v)-2\mathcal{C}(\bar{x})}{\eta^{1+2s}}\,d\eta\label{compareconeseq}\\
&\ \ \ \ \ +\int_r^\infty\frac{u(\bar{x}+\eta v)+u(\bar{x}-\eta v)-2u(\bar{x})}{\eta^{1+2s}}\,d\eta\notag\\
&\leq\int_{0}^r\frac{\mathcal{C}(\bar{x}+\eta v)+\mathcal{C}(\bar{x}-\eta v)-2\mathcal{C}(\bar{x})}{\eta^{1+2s}}\,d\eta\notag\\
&\ \ \ \ \ +\int_r^\infty\frac{\mathcal{C}(\bar{x}+\eta v)+\mathcal{C}(\bar{x}-\eta v)-2\mathcal{C}(\bar{x})}{\eta^{1+2s}}\,d\eta=0,\notag
\end{align}
where $v\in S^{N-1}$ denotes the direction of $\nabla \mathcal{C}(\bar{x})\neq 0$.  This chain of inequalities implies
\begin{align}
&\int_r^\infty\frac{u(\bar{x}+\eta v)+u(\bar{x}-\eta v)-2u(\bar{x})}{|\eta|^{1+2s}}\,d\eta\notag\\
&\ \ \ \ \ \ \ \ \ \ =\int_r^\infty\frac{\mathcal{C}(\bar{x}+\eta v)+\mathcal{C}(\bar{x}-\eta v)-2\mathcal{C}(\bar{x})}{|\eta|^{1+2s}}\,d\eta=0.\notag
\end{align}
Using $\mathcal{C}(\bar{x})\geq u(\bar{x})$ with equality at $\bar{x}$ we conclude $u=\mathcal{C}$ for all $x$ on the ray $\{\bar x+\eta v\}_{\eta\geq r}$,
contradicting $u=f$ in $\R^N\setminus \Omega$.

Thus $\bar x=x_0$, which implies
$$
u(x) \leq \mathcal C(x)=u(x_0)+(A+\epsilon)|x-x_0|^{2s-1}\qquad \forall\,x\in \R^N.
$$
Letting $\epsilon \to 0$ we finally obtain \eqref{eq:Holder}, which concludes the proof.
\end{proof}

The following result can be thought as the analogous of the absolutely minimizing property of infinity harmonic functions \cite{MR1861094,MR1218686}.

\begin{corollary}
\label{cor:Holder sol}
Let $u$ be a solution of $\I$ in the sense of Definition \ref{subsupdefn} inside a bounded open set $\Omega$, with $u=f$ on $\R^N\setminus \Omega$, $f\in C^{0,2s-1}(\R^N\setminus \Omega)$.
Then $u\in C^{0,2s-1}(\R^N)$ and
$$
[u]_{C^{0,2s-1}(\R^N)} \leq [f]_{C^{0,2s-1}(\R^N\setminus \Omega)}.
$$
\end{corollary}

\begin{proof}
We want to show that \eqref{eq:control f} holds with $C_0=\|u\|_{L^\infty(\Omega)}$
and  $A=[f]_{C^{0,2s-1}(\R^N\setminus \Omega)}$.  Since $u$ is a solution it is continuous and $C_0<\infty$.

Observe that, for any $z\in \R^N\setminus \Omega$ and $B \in\R$, the cone
$$
\mathcal{C}_z(x)=f(z)-A|x-z|^{2s-1}
$$
solves $\I \mathcal{C}_z=0$ inside $\Omega$, and $\mathcal C_z \leq f$ on $\R^N\setminus \Omega$ (by the definition of $[f]_{C^{0,2s-1}(\R^N\setminus \Omega)}$).
So, we can apply Theorem \ref{comparison} to conclude that $\mathcal C_z\leq u$, so \eqref{eq:control f} follows by the arbitrariness of $z\in \R^N\setminus \Omega$.
\end{proof}

\subsection{Stability}
The goal of this section is to show that the condition of being a ``(sub/super)solution at non-zero gradient points'' (see the end of Section \ref{derivation} for the definition)
is stable under uniform limit.  First we establish how subsolutions and supersolutions can be combined.
\begin{lemma}\label{maxsubissub}
The maximum [resp. minimum] of two subsolutions [resp. supersolutions] is a subsolution [resp. supersolution]. 
\end{lemma}
\begin{proof}
Let $u_1$ and $u_2$ be subsolutions, we argue $w=u_1\vee u_2$ is a subsolution.  Let $x_0\in \Omega$.
If $\phi\in C^{1,1}(x_0)\cap BC(\mathbb{R}^N)$ touches $w$ from above at $x_0$ then it must either touch $u_1$ from above at $x_0$ or touch $u_2$ from above at $x_0$. Assuming
with no loss of generality that the first case happens, using the monotonicity properties of the integral in the operator we get
\begin{align}
\I\tilde{w}(x_0)\geq \I\tilde{u}_1(x_0)\geq 0,\notag
\end{align}
where $\tilde{w}$ and $\tilde{u}$ are described by (\ref{subsupcut}). The statement about supersolutions is argued the same way.
\end{proof}

\begin{theorem}\label{stability}
Let $\Omega\subset \mathbb{R}^N$ and $u_n$ be a sequence of ``subsolutions [resp. supersolutions] at non-zero gradient points''.
Assume that
\begin{itemize}
\item $u_n$ converges to a function $u_0$ uniformly,
\item there exist $\alpha <2s$, $C>0$, such that $|u_n(x)| \leq C(1+|x|)^\alpha$ for all $n$.
\end{itemize}
Then $u_0$ is a ``subsolution [resp. supersolution] at non-zero gradient points''.
\end{theorem}
\begin{proof}
We only prove the statement with subsolutions.

Let $\phi\in C^{1,1}(x_0)$ touch $u_0$ from above at $x_0$, with $\nabla\phi(x_0)\neq 0$, and $\phi>u_0$ on $B_r(x_0)\setminus \{x_0\}$.
Since $u_n$ converges to a function $u_0$ locally uniformly, for $n$ sufficiently large there exists a small constant $\delta_n$ such that
$\phi+\delta_n$ touches $u_n$ above at a point $x_n \in B_r(x_0)$. Define $r_n=r-|x_n - x_0|$. Observe that $x_n \to x_0$ as $n \to \infty$, so $r_n \to r$ as $n \to \infty$.
Define
\begin{align}
\tilde{u}_n=\left\{
\begin{array}{lcl}
\phi(x)+\delta_n &\text{if}& |x-x_n|<r_n,\\
u_n(x) &\text{if}& |x-x_n|\geq r_n.
\end{array}\right.\notag
\end{align}
Since $\nabla\phi(x_0)\neq 0$, taking $n$ large enough we can ensure $\nabla\phi(x_n)\neq 0$. 
So, since $u_n$ is a subsolution at non-zero gradient points, $\I\tilde{u}_n(x_n)\geq 0$. Let $v_n\in S^{N-1}$
denote the direction of $\nabla\phi(x_n)$ and $v_0$ denote the direction of $\nabla\phi(x_0)$. We have
\begin{align}
0&\leq \int_0^{r_n}\frac{\phi(x_n+\eta v_n)+\phi(x_n-\eta v_n)-2\phi(x_n)}{\eta^{2s+1}}\,d\eta\notag\\
&\ \ \ \ \ \ +\int_{r_n}^\infty\frac{u_n(x_n+\eta v_n)+u_n(x_n-\eta v_n)-2u_n(x_n)}{\eta^{2s+1}}\,d\eta.\notag
\end{align}
By the $C^{1,1}$ regularity of $\phi$, the integrand in the first integral on the right hand side is bounded by the integrable function $M\eta^{1-2s}$. 
By the assumption on the growth at infinity of $u_n$,
also the integrand in the second integral on the right hand side is bounded, independently of $n$, by an integrable function.
Finally, also $v_n\rightarrow v_0$ (as $x_n \to x_0$). Hence, by the local uniform convergence of $u_n$ and applying the dominated convergence theorem, we find
\begin{align} 
0&\leq \int_0^r\frac{\phi(x_0+\eta v_0)+\phi(x_0-\eta v_0)-2\phi(x_0)}{\eta^{2s+1}}\,d\eta\notag\\
&\ \ \ \ \ \ +\int_r^\infty\frac{u_0(x_0+\eta v_0)+u_0(x_0-\eta v_0)-2u_0(x_0)}{\eta^{2s+1}}\,d\eta\notag\\
&=\I\tilde{u}_0(x_0).\notag
\end{align}
This proves $\I u_0(x)\geq 0$, as desired. 
\end{proof}

\subsection{Improved Regularity}\label{sect:improved regularity}

The aim of this section is to establish a Liouville-type theorem which will allow to show that solutions belong to $C^{0,2s-1}_0$ (see Definition \ref{def:holder}). 
The strategy is similar to the blow-up arguments employed in \cite{MR1804769,MR1861094}.

\begin{lemma}\label{impreg1}
Let $u\in C^{0,2s-1}(\mathbb{R}^N)$ be a global ``solution at non-zero gradient points''.
Then $u$ is constant.
\end{lemma}
\begin{proof}
Let $M=[u]_{C^{0,2s-1}(\mathbb{R}^N)}$.
Our goal is to prove $M=0$. By way of contradiction, assume $M> 0$. Then, there are two sequences $x_n\neq y_n$ such that 
\begin{align}
|u(x_n)-y(x_n)|>\left(M-\frac{1}{n}\right)|x_n-y_n|^{2s-1}.\notag
\end{align}
The assumptions of this lemma are preserved under translations, rotations, and the scaling $u(x)\rightarrow \lambda^{1-2s} u(\lambda x)$ for any $\lambda>0$.
Therefore, if $R_n:\R^N\to \R^N$ is a rotation such that $R_ne_1=\frac{y_n-x_n}{|y_n-x_n|}$, then the sequence of functions
$$
u_n(x)= |y_n-x_n|^{1-2s}\bigl[u(x_n+|y_n-x_n|R_nx)-u(x_n)\bigr]
$$
satisfies the assumptions of the lemma, and $u_n(e_1)>M-1/n$.  Then, Arzel\`a-Ascoli Theorem
gives the existence of a subsequence (which we do not relabel) and a function $u_0$ such that $u_n\rightarrow u_0$ uniformly on compact sets.
Moreover $u_0(0)=0$, $u_0(e_1)=M$, and Theorem \ref{stability} shows $u_0$ satisfies the assumptions of the lemma.

Let us observe that the cusp $\mathcal{C}^+=M|x|^{2s-1}$ touches $u_0$ from above at $e_1$, while $\mathcal{C}^-=-M|x-e_1|^{2s-1}+M$ touches $u_0$ from below at $0$.
Since $\nabla\mathcal{C}^+(e_1)\neq 0$ we may use it as a test function and arguing as in (\ref{compareconeseq})
we conclude $u=\mathcal{C}^+$ along the $e_1$ axis.  Similarly $\mathcal{C}^-$ touches $u$ from below at $0$, so we have also $u=\mathcal{C}^-$ along the $e_1$ axis.
However this is impossible since $\mathcal{C}^+\neq \mathcal{C}^-$ along the $e_1$ axis.
\end{proof}

\begin{corollary}
\label{cor:improved holder}
Let $\Omega\subset \R^N$ be open, and let $u\in C^{0,2s-1}(\Omega)$ ``solution at non-zero gradient points'' inside $\Omega$. Then $u\in C_{0,{\rm loc}}^{0,2s-1}(\Omega)$.
\end{corollary}
\begin{proof}
We have to prove that, for any $\Omega'\subset \Omega$ with $d(\Omega',\partial\Omega)>0$,
\begin{align}
\lim_{\delta \to 0} \sup_{x,y \in\Omega',\, |x-y|\leq \delta}\frac{|f(x)-f(y)|}{|x-y|^{2s-1}}=0.\notag
\end{align}
Assume by contradiction this is not the case. Then we can find as sequence of points $x_n\neq y_n \in \Omega'$, with $|x_n - y_n| \leq 1/n$, such that
$$
\frac{|u(x_n)-u(y_n)|}{|x_n-y_n|^{2s-1}} \geq c_0 >0.
$$
Let us define the sequence of functions
$$
u_n(x)=|y_n-x_n|^{1-2s}\bigl[u(x_n+|y_n-x_n|R_nx)-u(x_n)\bigr],
$$
where $R_n:\R^N\to \R^N$ is a rotation such that $R_ne_1=\frac{y_n-x_n}{|y_n-x_n|}$.
Then $u_n$ are ``solution at non-zero gradient points'' inside the ball $B_{\delta/|y_n - y_n|}$ (since $B_\delta(x_n)\subset \Omega$). Moreover,
\begin{equation}
u_n(0)=0, \qquad u_n(e_1)\geq c_0. 
\end{equation}
Let us observe that $\frac{d(\Omega',\partial\Omega)}{|y_n - y_n|}\to \infty$ as $n \to \infty$.
So, as in the proof of Lemma \ref{impreg1},  by the uniform H\"older continuity of $u_n$ we can combine Arzel\`a-Ascoli Theorem with the stability
Theorem \ref{stability} to extract a subsequence with limit $u_0$ which satisfies the assumptions of Lemma \ref{impreg1}. Hence $u_0$ is constant,
which is impossible since $u_0(0)=0$ and $u_0(e_1) \geq c_0$.
\end{proof}

\section{A Monotone Dirichlet Problem}\label{monotone}
For the ``$i$th'' component of $x$ write $x_i$ and use $e_i$ to denote the unit vector in the ``$i$'' direction.
We take $\Omega$ as a strip orthogonal to the $e_1$ axis described as follows:
Consider two maps $\Gamma_1,\Gamma_2:\mathbb{R}^{N-1}\rightarrow \mathbb{R}$ which define the boundaries
\begin{align}
\partial\Omega^{-}&=\{(\Gamma_1(\hat{x}),\hat{x}_1,...,\hat{x}_{n-1}):\hat{x}\in \mathbb{R}^{N-1}\},\notag\\
\partial\Omega^{+}&=\{(\Gamma_2(\hat{x}),\hat{x}_1,...,\hat{x}_{n-1}):\hat{x}\in \mathbb{R}^{N-1}\}.\notag
\end{align}
We assume $\Gamma_1$ and $\Gamma_2$ are $C^{1,1}$ and uniformly separated. More precisely:
\begin{itemize}
\item There are constants $M>m>0$ such that, for all $\hat{x}\in \mathbb{R}^{N-1}$,
\begin{align}
\begin{array}{c}
0\leq \Gamma_1(\hat{x})  \leq \Gamma_2(\hat{x}) \leq M,\\
\Gamma_2(\hat{x})  -\Gamma_1(\hat{x}) \geq m.
\end{array}\label{eq:m M}
\end{align}
\item There exists a constant $C_1$ such that 
\begin{align}
\sup_{\hat{x}\in \mathbb{R}^{N-1}}|\partial_k\Gamma_i(\hat{x})|+|\partial_l\partial_k\Gamma_i(\hat{x})|\leq C_1\label{lipass}
\end{align}
for each $i=1,2$ and $k,l=1,...,n-1$.  
\end{itemize}
We also use the notation 
\begin{align}
\Omega^{c,-}=\{x\in \mathbb{R}^N : x_1\leq \Gamma_1(x_2,..., x_n)\},\notag\\
\Omega^{c,+}=\{x\in \mathbb{R}^N : x_1\geq \Gamma_2(x_2,..., x_n)\},\notag
\end{align}
to denote the two connected components of $\R^N\setminus \overline\Omega$.

\begin{figure}
\centerline{\epsfysize=2.5truein\epsfbox{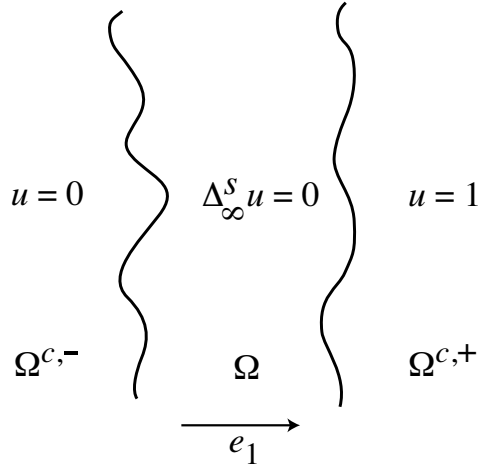}}
\caption{The monotone Dirichlet problem.}
\label{fig2}
\end{figure}

Consider the problem
\begin{align}
\left\{\begin{array}{cclcl}
\I u(x)&=&0 &\text{if}  &x\in \Omega,\\
u(x)&=&1 &\text{if} &x\in \Omega^{c,+},\\
u(x)&=&0 &\text{if} &x\in \Omega^{c,-}.
\end{array}\right.\label{PDEm}
\end{align}
Following Perron's method, we will show the supremum of subsolutions is a solution for (\ref{PDEm}) in the sense of Definition \ref{subsupdefn}.

More precisely, consider the family of subsolutions $\mathcal{F}$ given by
\begin{align}
u\in \mathcal{F} \ \ \ \ \ \text{if} \ \ \ \ \ 
\left\{\begin{array}{cclcl}
\I u(x)&\geq&0 &\text{if}  &x\in \Omega,\\
u(x)&\leq&0 &\text{if} &x\in \Omega^{c,-},\\
u(x)&\leq&1 &\text{if}  &x\in \Omega\cup\Omega^{c,+}.
\end{array}\right.\label{monotonefamily}
\end{align}
(Recall that, by Definition \ref{subsupdefn}, the set of functions in $\mathcal{F}$ are continuous inside $\Omega$ by assumption.)
The function $u\equiv 0$ belongs to $\mathcal F$, so the family is not empty. Moreover, every element of $\mathcal F$ is bounded above by $1$. So, we define
our solution candidate
\begin{align}
U(x)=\sup\{u(x): u\in\mathcal{F}\}.\notag
\end{align}
Let us remark that, since the indicator function of $\Omega^{c,+}$ belongs to $\mathcal F$, we have
\begin{equation}
\label{eq:bdry U}
U=0 \quad \text{in } \Omega^{c,-}, \qquad U=1 \quad \text{in } \Omega^{c,+}.
\end{equation}
The remainder of this section is devoted to proving the following theorem (recall Definition \ref{def:holder}):
\begin{theorem}\label{PDEmthrm}
Let $\Omega$ satisfy the above assumptions.  Then $U(x) \in C^{0,2s-1}(\R^N)\cap C_{0,loc}^{0,2s-1}(\Omega)$, and it is the unique solution to the problem (\ref{PDEm}).
\end{theorem}
This theorem is a combination of Lemma \ref{Usubsolnlemma} and Theorems \ref{monoextthrm} and \ref{monocompthrm} which are proved in the following sections.
The strategy of the proof is to find suitable subsolutions and supersolutions to establish specific growth and decay rates of $U$, which will in turn imply a uniform monotonicity in the following sense:
\begin{definition}\label{unifmonotonedefn}
We say a function $u:\Omega\rightarrow \mathbb{R}$ is uniformly monotone in the $e_1$ direction with exponent $\alpha>0$ and constant $\beta>0$ if the following statement holds:
For every $x\in\Omega$ there exists $h_0>0$ such that 
\begin{align}
u(x)+\beta h^\alpha\leq u(x+e_1h)\qquad \forall\, 0 \leq h \leq h_0.\notag
\end{align}
\end{definition}

By constructing suitable barriers we will show $U$ grows like $d^s(x,\partial\Omega^-)$ near $\partial\Omega^-$,
and decays like $1-d^s(x,\partial\Omega^+)$ near $\partial\Omega^+$.  This sharp growth near the boundary influences the solution in the interior in such a way that $U$
is uniformly monotone with exponent $\alpha=1+s$ away from $\partial\Omega$. As shown in Lemma \ref{notouchlem}, this uniform monotonicity implies we can only touch $U$ by test
functions which have a non-zero derivative. 
Thanks to this fact, the operator $\I$ will be stable under uniform limits (see Theorem \ref{stability}). This will allow to prove that
$U$ is a solution to the problem. 

Concerning uniqueness, let us remark that $\Omega$ is not a compact set,
so we cannot apply our general comparison principle (see Theorem \ref{comparison}). However, in this specific situation
we will be able to take advantage of the fact that $\Omega$ is bounded in the $e_1$ direction, that our solution $U$ is uniformly monotone in that direction, and that
being a subsolution or supersolution is stable under translations, 
to show a comparison principle for this problem (Theorem \ref{monocompthrm}).  This comparison principle implies uniqueness of the solution.

\subsection{Basic Monotone Properties of $U$}\label{monosec}
Let $L$ denote the Lipschitz constant of $\Gamma_1$ and $\Gamma_2$ (see assumption (\ref{lipass})).  Set $\theta={\rm arccot}(L)\in (0,\frac{\pi}{2})$,
and  consider the open cone
\begin{align}
C^+=\left\{x\in\mathbb{R}^N\setminus \{0\}: \cos(\theta)< \frac{ x\cdot e_1}{|x|} \leq 1\right\}.\label{monocone}
\end{align}
If $x_0\in\partial\Omega^-$ the cone $x_0+C^+$ does not intersect $\partial\Omega^-$ except at $x_0$, that is $x_0+C^+\cap \Omega^{c,-}=\{x_0\}$.

The aim of this section is to show that $C^+$ defines the directions of monotonicity for $U$. More precisely, we want to prove the following:
\begin{proposition}
\label{prop:monotone}
$U(x+y)\geq U(x)$ for all $x \in \R^N$, $y \in C^+$. 
\end{proposition}
Thanks to this result, in the sequel, whenever $y\in S^{N-1}\cap C^+$ (that is, $\cos(\theta)<y\cdot e_1\leq 1$),
we will say that such direction lies ``in the cone of monotonicity for $U$''.

\begin{proof}
Fix $y\in C^+$. We want to show that
$$
U(x+y) \geq U(x)\qquad \forall \, x \in \R^N.
$$
Thanks to \eqref{eq:bdry U} and the $L$-Lipschitz regularity of $\Gamma_1$ and $\Gamma_2$,
it suffices to prove the results for $x\in \Omega$. Fix $x_0\in \Omega$,
and let $u^\delta\in\mathcal{F}$ be a sequence of subsolutions such that $U(x_0)\leq u^\delta(x_0)+\delta$, with $\delta\to 0$.
Since increasing the value of $u^\delta$ outside $\Omega$ increases the value of $\I u^\delta$ inside $\Omega$, we may assume the subsolutions satisfy
\begin{align}
u^\delta(x)=\left\{
\begin{array}{lcl}
1 &\text{if} &x\in \Omega^{c,+},\\
0 &\text{if} &x\in \Omega^{c,-}.
\end{array}\right.\notag
\end{align}

\begin{figure}
\centerline{\epsfysize=2.5truein\epsfbox{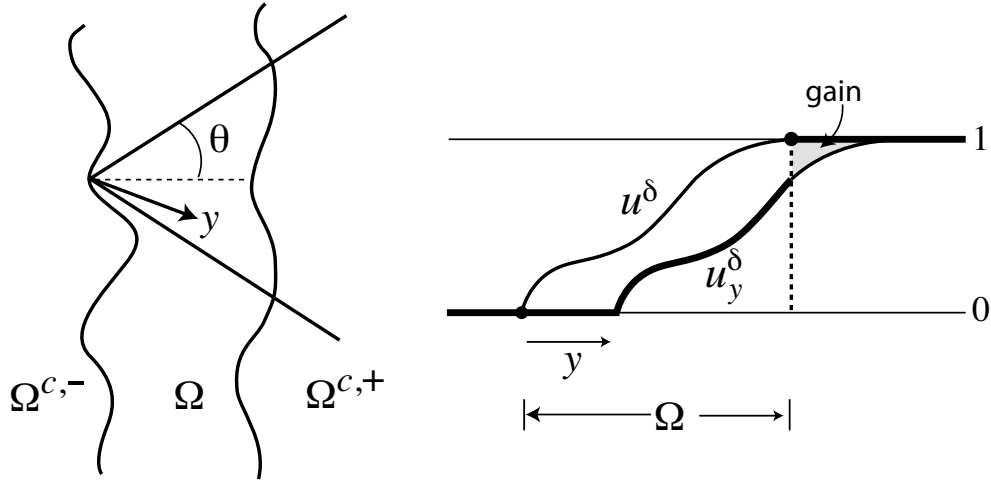}}
\caption{Translation of a subsolution.}
\label{fig3}
\end{figure}

Then, we claim that the function
\begin{align}
u^\delta_y(x)=\left\{
\begin{array}{lcl}
u^\delta(x-y) &\text{if}  &x\in \overline\Omega,\\
1 &\text{if} &x\in \Omega^{c,+},\\
0 &\text{if} &x\in \Omega^{c,-},
\end{array}\right. \label{uhdefn}
\end{align}
is also a subsolution.
To check this, first note $u^\delta(x-y)\leq u^\delta(x)$ for all $x\in \R^N\setminus \Omega$ (since $y \in C^+$ and by the Lipschitz assumption for the boundary),
so $u^\delta_y \geq u^\delta(\cdot-y)$ and they coincide inside $\Omega$. Hence, since
$u^\delta$ is a subsolution and the value of $\I$ increase when increasing the value of the function outside $\Omega$,
we get $\I u^\delta_y (x)\geq \I u^\delta(x)\geq 0$ for all $x\in \Omega$ (see Figure \ref{fig3}).
Hence, $u^\delta_{h}$ is a subsolution as well, so it must be less then $U$ everywhere in $\Omega$. This gives 
\begin{align}
U(x_0)\leq u^\delta(x_0)+\delta=u_h^\delta(x_0+hy)+\delta\leq  U(x_0+hy)+\delta,\notag
\end{align}
and the result follows by letting $\delta \to 0$.
\end{proof}
The monotonicity of $U$ has an immediate implication on the possible values for the gradient of test functions which touch $U$:
if $\phi$ is a test function which touches $U$ from above at a point $x_0\in \Omega$, then either $\nabla \phi(x_0)=0$ or $\cos(\frac{\pi}{2}-\theta)\leq v\cdot e_1\leq 1$,
where $v\in S^{N-1}$ is the direction of $\nabla\phi(x_0)$. To see this, let $y\in S^{N-1}$ be in the cone of monotonicity. Then
\begin{align}
\nabla\phi(x_0)\cdot y=&\lim_{h\rightarrow 0} \frac{\phi(x_0-hy)-\phi(x_0)}{h}\notag\\
&\geq \limsup_{h\rightarrow 0} \frac{U(x_0-hy)-U(x_0)}{h}\geq 0.\notag
\end{align}
This implies the angle between $\nabla\phi(x_0)$ and $y$ is less then $\frac{\pi}{2}$ and in turn, since $y$ was arbitrary in the cone of monotonicity,
the angle between $\nabla\phi(x_0)$ and $e_1$ is less then $\frac{\pi}{2}-\theta$.  A similar argument holds for test functions touching $U$ from below.

The angle $\frac{\pi}{2}-\theta$ is important in our analysis, so from here on we denote
\begin{equation}
\label{eq:Ctheta}
C_\theta=\cos\left(\frac{\pi}{2}-\theta\right)>0.
\end{equation}
In the sequel, we will also use the cones opening in the opposite direction:
\begin{align}
C^-=-C^+=\left\{x\in\mathbb{R}^N: -1\leq \frac{x\cdot e_1}{|x|} < -\cos(\theta)\right\}\notag
\end{align}

\subsection{Barriers and Growth Estimates}
Now, we want to construct suitable barriers to show that $U$ detaches from $\partial \Omega^-$ at least as $d^s(x,\partial\Omega^-)$,
and from  $\partial \Omega^+$ at least as $d^s(x,\partial\Omega^+)$.

This $d^s$ growth is naturally suggested by the fact that the function
$(x_1^+)^{s}=(x_1\vee 0)^s:\mathbb{R}^N\rightarrow \mathbb{R}^+$ solves  $\I(x_1^+)^{s}=0$ for $x_1>0$.
Indeed, this follows immediately from the fact that $\R\ni \eta \mapsto (\eta^+)^s$ solves the $s$-fractional Laplacian on the positive half-line in one dimension:
$\Delta^s (\eta^+)^{s}=0$ on $(0,\infty)$
(see \cite[Propositions 5.4 and 5.5]{MR2367025}).

Our goal is to show the existence of a small constant $\epsilon>0$ such that the function
\begin{align}
g_\epsilon(x)&=\left\{
\begin{array}{lcl}
\epsilon d^{s}(x,\partial\Omega^-) &\text{if}  &x\in \overline\Omega,\\
1 &\text{if} &x\in \Omega^{c,+},\\
0 &\text{if} &x\in \Omega^{c,-}.
\end{array}\right. \label{growth2barriersub}
\end{align}
is a subsolution near $\partial \Omega^-$. The reason why this should be true is that, by the discussion above,
$$
\I\bigl( d^{s}(\cdot,\partial\Omega^-)\bigr)(x_0)=0
$$ near
$\partial \Omega^-$ for $x_0 \in \Omega$ (here we use that $\partial \Omega^-$ is $C^{1,1}$).
Now, when evaluating the integral $\I g_\epsilon(x_0)$  for some $x_0 \in \Omega$,
this integral will differ from $\I(\epsilon d^{s}(\cdot,\partial\Omega^-))(x_0)$ in two terms: if $x \in \Omega^{c,+}$ and $\epsilon d^{s}(x,\partial\Omega^-) \leq 1$
then we ``gain'' inside the integral since $g_\epsilon(x)=1 \geq \epsilon d^{s}(x,\partial\Omega^-)$. On the other hand,
if $x \in \Omega^{c,+}$ and $\epsilon d^{s}(x,\partial\Omega^-) \geq 1$
then we have a ``loss''. So, the goal becomes to show that the gain compensate the loss for $\epsilon$ sufficiently small.

This argument is however not enough to conclude the proof on the growth of $U$, since $d^{s}(x,\partial\Omega^-)$ is a solution only in a neighborhood of $\partial \Omega^-$
(which depends on the $C^{1,1}$ regularity of $\partial \Omega^-$).
In order to handle this problem, we first show $U$ grows like $d^{2s}(x,\partial\Omega^-)$ inside $\Omega$,
so for $\epsilon$ small we will only need to consider (\ref{growth2barriersub}) near the boundary.  This is the content of the next lemma.

\begin{lemma}\label{growth1}
There is a constant $C>0$ such that $Cd^{2s}(x,\partial\Omega^-) \leq U(x)\leq 1-Cd^{2s}(x,\partial\Omega^+)$.
\end{lemma}
We will prove the lower bound on $U$ by constructing a subsolution obtained as an envelope of paraboloids. The
construction is contained in the following lemma.  We use the notation
$\partial\Omega^+_{t_0}=\{x-t_0e_1:x\in\partial\Omega^+\}$ and $\partial\Omega^-_{t_0}=\{x+t_0e_1:x\in\partial\Omega^-\}$.

\begin{lemma}\label{growth1construction}
There is a small constant $c_0\in (0,1)$, depending only on $s$ and the geometry of the problem
such that, for any set $\mathcal{S}\subset\mathbb{R}^N$ and constants $A>0$, $t_0 \in (0,M]$ ($M$ as in \eqref{eq:m M}) satisfying $A C_\theta^{2-2s}t_0^{2-2s}\leq c_0$
($C_\theta$ as in \eqref{eq:Ctheta}), the function 
\begin{align}
P^+(x)&=\left\{
\begin{array}{lcl}
P_{\mathcal{S}}^+(x) &\text{if}  &x\in \overline\Omega,\\
1 &\text{if} &x\in \Omega^{c,+},\\
0 &\text{if} &x\in \Omega^{c,-},
\end{array}\right. \label{growth1barriersub}\\
P^+_{\mathcal{S}}(x)&=\sup_{x_0\in \partial\Omega^-_{t_0}\cap\mathcal{S}}\left(\sup_{x' - x_0 + t_0 e_1 \in C^+}\left\{-A(|x-x'|^2-C_\theta^2t_0^2)\vee 0\right\}\right)\notag
\end{align}
is a subsolution. Likewise,
\begin{align}
P^-(x)&=\left\{
\begin{array}{lcl}
P^-_{\mathcal{S}}(x) &\text{if}  &x\in \overline\Omega,\\
1 &\text{if} &x\in \Omega^{c,+},\\
0 &\text{if} &x\in \Omega^{c,-},
\end{array}\right. \label{growth1barriersup}\\
P^-_{\mathcal{S}}(x)&=\inf_{x_0\in \partial\Omega^+_{t_0}\cap \mathcal{S}}\left(\inf_{x'-x_0 - t_0 e_1\in C^-}\left\{1+A(|x-x'|^2-C_\theta^2t_0^2)\wedge 1\right\}\right)\notag
\end{align}
is a supersolution.
\end{lemma}

\begin{figure}
\centerline{\epsfysize=2.5truein\epsfbox{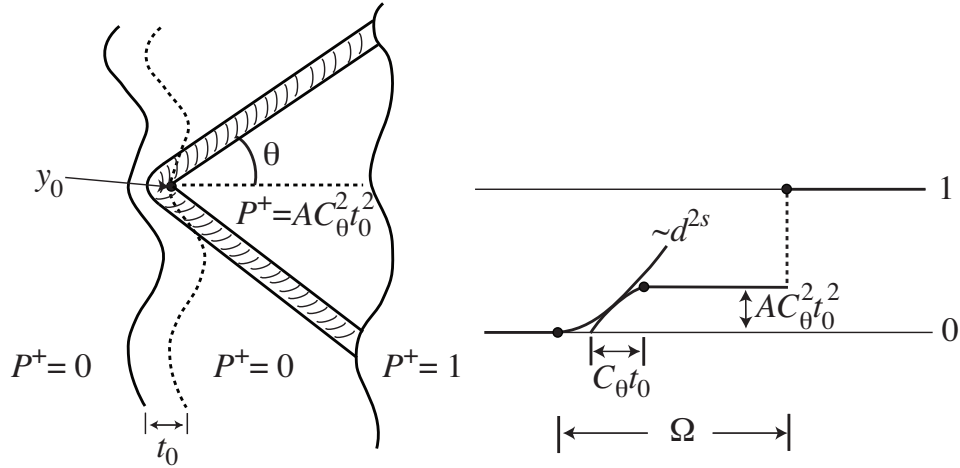}}
\caption{The barrier $P^+$ for $\mathcal{S}=\{y_0\}$.}
\label{fig4}
\end{figure}

The idea in the construction of $P^+$ (see Figure \ref{fig4}) is that we will calculate
how high we can raise a paraboloid $-A(|x-x'|^2-C_\theta^2t_0^2)$
near the boundary $\partial\Omega^-$, and have a subsolution using the boundary value $1$ inside $\Omega^{c,+}$ to compensate the concave shape of the paraboloid when computing the operator.
In order to take advantage of the value inside $\Omega^{c,+}$, we have to ensure that the derivative of our
test function is always inside $C^+$ or $C^-$. This is the reason to construct our barrier as a supremum of these paraboloids in the cone of monotonicity
(indeed, the gradient of $P_{\mathcal{S}}^+$ always points inside $C^+$).
Finally, having a general set $\mathcal{S}$ allows us to construct more general barriers of this form, so we can use them locally or globally, as needed.

Before proving the lemma, we estimate the value of $\I$ for a ``cut'' paraboloid $-A(|x-x_0|^2-r_0^2)\vee 0$.
The following result follows by a simple scaling argument. We leave the details to the interested reader.
\begin{lemma}\label{lemma:I parab}
For any $s \in (1/2,1)$ there exists a constant $C_s$, depending only on $s$, such that
$$
\I [-A(|x-x_0|^2-r_0^2)\vee 0]\geq -C_sAr_0^{2-2s}.
$$
\end{lemma}

\begin{proof}[Proof of Lemma \ref{growth1construction}]
We will prove $P^+$ is a subsolution.  Showing $P^-$ is a supersolution follows a similar argument.  If $x\in \Omega$ is such that $P^+(x)= 0$, then $\I P^+(x)\geq 0$
since $P^+\geq 0$. So, it remains to check $\I P^+(x)\geq 0$ if $P^+(x)\neq 0$. 
We estimate by computing lower bounds for the positive and negative contributions to the operator.  

Let $r_0=C_\theta t_0$.  For the positive part we estimate from below the contribution from $\Omega^{c,+}$. Since $\nabla P^+ \in C^+$,
the worst case is when the direction $v$ in the integral satisfies $C_\theta= v\cdot e_1$. So, since $d(x,\partial\Omega^+)\leq M$, the contribution from $\Omega^{c,+}$ is greater
or equal than
\begin{align}
\int_{M/C_\theta}^\infty\frac{1-Ar_0^2}{\eta^{1+2s}}\,d\eta = \frac{1-Ar_0^2}{2s}\left(\frac{C_\theta}{M}\right)^{2s}.\notag 
\end{align}
To estimate the negative contribution, we use Lemma \ref{lemma:I parab} to get that
\begin{align}
\inf_{x'\in B_{r_0}(x_0)} \left\{\I[-A(|x-x_0|^2-r_0^2)\vee 0](x')\right\} \geq -C_s Ar_0^{2-2s}.
\end{align}
bounds the negative contribution from below. Hence, all together we have
\begin{align}
\I P(x)\geq \frac{1-Ar_0^2}{2s}\left(\frac{C_\theta}{M}\right)^{2s}-C_s Ar_0^{2-2s}.\notag
\end{align}
Since $r_0 \leq t_0 \leq M$, it is easily seen that there exists a small constant $c_0$, depending only on $s$ and the geometry of the problem,
such that $\I P^+(x)\geq 0$ if $Ar_0^{2-2s} =A C_\theta^{2-2s}t_0^{2-2s}\leq c_0$.
\end{proof}

\begin{proof}[Proof of Lemma \ref{growth1}]
For any point $x_0\in \Omega$ there exists $t_0 \in (0,M]$ such that $x_0\in\partial\Omega^-_{t_0}$.
It is clear that $t_0\geq d(x_0,\partial\Omega^-)$, so it suffices to prove $U(x_0)\geq C t_0^{2s}$ for some constant $C>0$ independent of $x_0$.
Lemma \ref{growth1construction} gives the existence of a constant $c_0$, independent of $x_0$, so that if $AC_\theta^{2-2s}t_0^{2-2s}\leq c_0$
then $P^+(x)$ given by (\ref{growth1barriersub}) with $\mathcal{S}=\{x_0\}$ is a subsolution, and therefore $U(x)\geq P^+(x)$.
In particular $U(x_0)\geq P^+(x_0)=AC_\theta^2t^2_0$, so choosing $A=c_0C_\theta^{2s-2}t_0^{2s-2}$ we get $U(x_0)\geq c_0C_\theta^{2s}t_0^{2s}$, as desired.

Similarly, $P^-(x)$ given by (\ref{growth1barriersup}) with $\mathcal{S}=\{x_0\}$ is a supersolution which is equal to $1$ for all values of $x\in \Omega$ outside of a compact subset.
Using the comparison principle on compact sets (Theorem \ref{comparison}) we conclude $u\leq P^-$ for any $u\in \mathcal{F}$,
so the same must be true of the pointwise supremum of the family $\mathcal{F}$, and we conclude as above.
\end{proof}

\begin{lemma}\label{growth2}
There is a constant $\epsilon>0$ such that $\epsilon d^{s}(x,\partial\Omega^-)\leq U(x)\leq 1-\epsilon d^{s}(x,\partial\Omega^+)$ inside $\Omega$.
\end{lemma}

\begin{proof}
We will prove the statement on the growth near $\partial\Omega^-$ by constructing a subsolution which behaves like $d^{s}(x,\partial\Omega^-)$ near the boundary. 

The uniform $C^{1,1}$ assumption for the boundary implies the existence of a neighborhood $\N$ of $\partial\Omega^-$ such that:
\begin{enumerate}
\item[(a)] for $t_0>0$ small enough, $x+2t_0e_1\in N$ for all $x\in\partial\Omega^-$;
\item[(b)] for any $x\in \N$ the line with direction $\nabla d^{s}(x,\partial\Omega^-)$  and passing through $x$ intersects $\partial\Omega^-$ orthogonally.
\end{enumerate}
Fix $t_0>0$ small so that (a) holds, and pick $A$ such that $AC_\theta^{2-2s}t_0^{2-2s}\leq c_0$, where $c_0$
is the constant given by Lemma \ref{growth1construction}.
Then the barrier $P^+(x)$ given by (\ref{growth1barriersub}) with $\mathcal{S}=\mathbb{R}^n$ is a subsolution.
Let $g_\epsilon$ be defined by (\ref{growth2barriersub}).  Choosing $\epsilon$ small enough we can guarantee $g_\epsilon(x) \leq P^+(x)$ for all $x\in \Omega\setminus \N$.
Hence, the growth estimate will be established once we show $w(x)=g_\epsilon(x)\vee P^+(x)$ is a subsolution (see Figure \ref{fig5a}).   

\begin{figure}
\centerline{\epsfysize=2truein\epsfbox{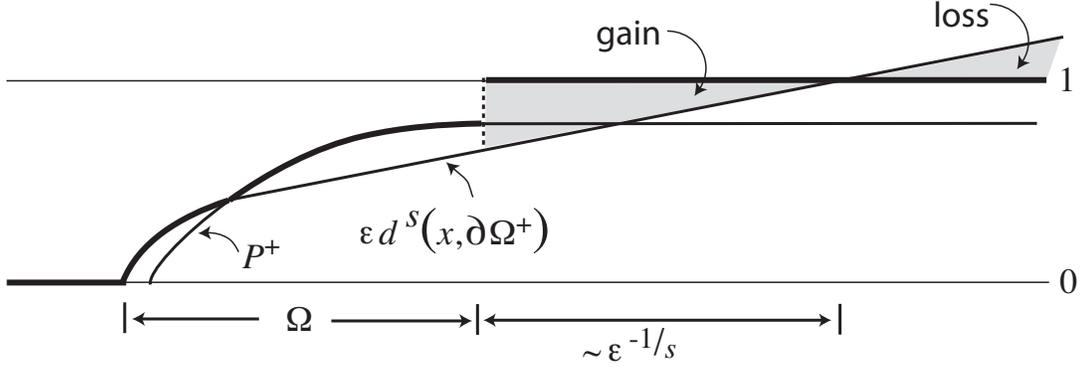}}
\caption{The barrier $w$ in the proof of Lemma \ref{growth1}.}
\label{fig5a}
\end{figure}

Consider any point $x\in \Omega$ and let $\phi\in C^{1,1}(x)$ be a test function touching $w$ from above at $x$.  If $P^+(x)\geq g_\epsilon(x)$ then
$\I\tilde{w}(x)\geq \I\tilde{P}^+(x)\geq 0$, where $\tilde{w}$ and $\tilde{P}^+$ are described by (\ref{subsupcut}). To conclude the proof,
it suffices to show $\I g_\epsilon(x)\geq 0$ when $g_\epsilon(x)>P^+(x)$ (in particular, $x \in \N$), since $\I\tilde{w}(x)\geq \I g_\epsilon(x)$ in this case.

The assumed regularity of the boundary $\partial\Omega^-$ allows us to compute $\I g_\epsilon(x)$ directly without appealing to test functions when $x\in \N$.  
Let $v\in S^{N-1}$ be the direction of $\nabla d(x,\partial\Omega^-)$. As $x\in \N$, by (b) above the line $\{x+\eta v\}_{\eta\in\R}$ intersects $\partial\Omega^-$ orthogonally.
Let $t>0$ [resp. $t'>0$] be the distance between $x$ and $\partial\Omega^-$ [resp. $\partial\Omega^+$] along this line (see Figure \ref{fig5b}.)
Replacing $\N$ with a possibly smaller neighborhood, we may assume $t<t'$.  Since $v$ lies in the cone of monotonicity, the maximum value for $t'$ is
$M/C_\theta$.  We further assume $\epsilon$ is small enough that $\epsilon^{-1/s}>2M/C_\theta$. Using the exact solution $(\eta^+)^{s}$ to the one dimensional
$s$-fractional Laplacian on the positive half-line \cite[Propositions 5.4 and 5.5]{MR2367025}
(that is, for all $\eta>0$,  $\int_{-\infty}^\infty \frac{((\eta+\tau)^+)^s - (\eta^+)^s}{|\tau|^{1+2s}}\,d\tau=0$ in the principal value sense) we have
\begin{align}
\I g_\epsilon(x)&\geq\int_{t'}^{\epsilon^{-\frac{1}{s}}-t}\frac{1-\epsilon(t+\eta)^s}{\eta^{1+2s}}\,d\eta -
\int_{\epsilon^{-\frac{1}{s}}-t}^\infty\frac{\epsilon(t+\eta)^s-1}{\eta^{1+2s}}\,d\eta.\notag
\end{align}
\begin{figure}
\centerline{\epsfysize=2.5truein\epsfbox{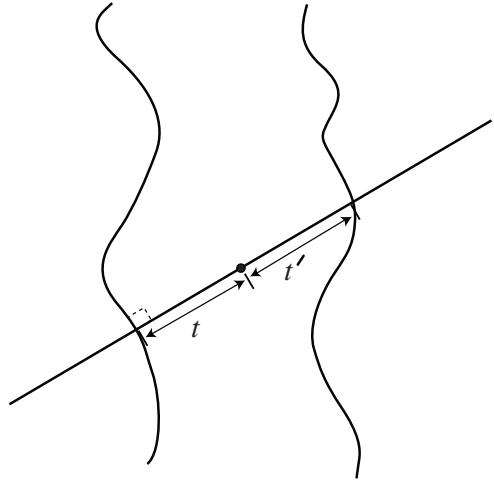}}
\caption{Measure of the distance to the boundary in the proof of Lemma \ref{growth1}.}
\label{fig5b}
\end{figure}
The first integral on the right hand side represents the ``gain'' where $\epsilon(t+\eta)^{s}\leq 1$ while the second integral represents the ``loss''
where $\epsilon(t+\eta)^{s}\geq 1$ (see Figure \ref{fig5a}).  We now show the right hand side is positive if $\epsilon$ is small enough.  For any $\eta>t'>t$
it holds $\frac{t+\eta}{\eta}\leq 2$. So, by choosing $\epsilon<C_\theta^2/(8M^s)$ and recalling that $t'<M/C_\theta<\epsilon^{-1/s}/2$ we get
\begin{align}
\int_{t'}^{\epsilon^{-\frac{1}{s}}-t}\frac{1-\epsilon(t+\eta)^s}{\eta^{1+2s}}\,d\eta&\geq \int_{t'}^{\epsilon^{-\frac{1}{s}}-t}\left(\frac{1}{\eta^{1+2s}}-\frac{\epsilon2^s}{\eta^{1+s}}\right)\,d\eta\notag\\
&\geq \frac{1}{s}\frac{1}{(t')^s}\left(\frac{1}{2}\frac{1}{(t')^s}-\epsilon 2^s\right)
 -\frac{1}{2s}\frac{1}{(\epsilon^{-\frac{1}{s}}-t)^{2s}}\notag\\
&\geq \frac{1}{4s}\frac{C_\theta^{2s}}{M^{2s}}-\frac{2^{2s-1}}{s}\epsilon^2.\notag
\end{align}
Also,
\begin{align}
\int_{\epsilon^{-\frac{1}{s}}-t}^\infty\frac{\epsilon(t+\eta)^s-1}{\eta^{1+2s}}\,d\eta&\leq \int_{\epsilon^{-\frac{1}{s}}-t}^\infty\left(\frac{2^s\epsilon}{\eta^{1+s}}-
\frac{1}{\eta^{1+2s}}\right)\,d\eta\notag\\
&\leq \frac{1}{s}\frac{\epsilon2^s}{(\epsilon^{-\frac{1}{s}}-t)^s}\leq \frac{2^{2s}}{s}\epsilon^2.\notag
\end{align}
Hence, combining all together,
\begin{align}
\I g_\epsilon(x)\geq \frac{1}{4s}\frac{C_\theta^{2s}}{M^{2s}} -\frac{2^{2s}+2^{2s-1}}{s}\epsilon^2.\notag
\end{align}
From this we conclude that for $\epsilon$ sufficiently small (the smallness depending only on  $s$ and the geometry of the problem) we have $\I g_\epsilon(x)\geq 0$.
Hence $w$ is a subsolution, which implies $U \geq w$ and establishes the growth estimate from below.

To deduce the decay of $U$ moving away from the boundary $\partial\Omega^+$, one uses similar techniques as above to construct a supersolution
with the desired decay and then applies the comparison principle on compact sets (Theorem \ref{comparison}),
arguing essentially as in the last paragraph of the proof of Lemma \ref{growth1}.
The key difference is that the comparison needs to be done on a compact set. So, instead of using the equivalent of (\ref{growth2barriersub}) which would require comparison on a non-compact set one should construct barriers like
\begin{align}
g^\epsilon(x)=\left\{
\begin{array}{lcl}
1-\epsilon d^{s}(x,\partial\mathcal{P}) &\text{if}  &x\in \overline\Omega\cap\mathcal{P},\\
1 &\text{if} &x\in \Omega^{c,+}\cup(\Omega\setminus \mathcal{P}),\\
0 &\text{if} &x\in \Omega^{c,-}.
\end{array}\right. 
\end{align}
where $\mathcal{P}$ is a ``parabolic set'' touching $\partial\Omega^+$ from the left.
More precisely, at any point $\hat{x}\in\mathbb{R}^{N-1}$ the regularity of the boundary implies the existence
of a paraboloid $p:\mathbb{R}^{N-1}\rightarrow \mathbb{R}$, with uniform opening, which touches $\Gamma_2$ from below at $(\Gamma_2(\hat x),\hat{x})$.
Define $\mathcal{P}=\{x\in \mathbb{R}^N : x_1\leq p(x_2,..., x_n)\}$, and choose $\mathcal{S}$ in (\ref{growth1barriersup}) to be a large
ball centered at $(\Gamma_2(\hat x),\hat{x})$,
which contains $\Omega \cap \mathcal P$. In this way one constructs a family of supersolutions $g^\epsilon \wedge P^-$ (depending on the point $(\Gamma_2(\hat x),\hat{x})$)
which are equal to $1$ outside a compact subset of $\Omega$,
so that one can apply the comparison on compact sets. Since $\hat x$ is arbitrary, this proves the desired decay.
\end{proof}

\begin{remark}\label{rmk:growth}
{\rm 
Arguing similar to the above proof one can show the $d^s$ growth and decay proved in the previous lemma is optimal: there exists a universal constant $C>0$
such that
$$
1-C d^s(x,\partial \Omega^+)\leq U(x) \leq C d^s(x,\partial \Omega^-).
$$
Indeed, if $\mathcal{P}\subset\Omega^{c,-}$
is a ``parabolic set'' touching $\partial\Omega^-$ from the left, then the function
\begin{align}
g(x)=\left\{
\begin{array}{lcl}
A d^{s}(x,\partial\mathcal{P})\wedge 1 &\text{if}  &x\in \overline\Omega\cup\Omega^{c,+},\\
0 &\text{if} &x\in \Omega^{c,-}.
\end{array}\right. 
\end{align}
is a supersolution for $A>0$ large enough so that $Ad^{s}(x,\partial\mathcal{P})\geq 1$ for all $x\in \Omega^{c,+}$.
Also, $g\geq 1$ outside a compact subset of $\Omega$. So Theorem \ref{comparison} implies that any $u\in\mathcal{F}$ is bounded from above by $g$, and the same is true for $U$.
The decay is argued in an analogous way.}
\end{remark}

\subsection{Existence and Comparison}
We are now in a position to show $U$ is uniformly monotone. This will allow us to apply Perron's method to show that $U$ is a solution, and establish a comparison principle, proving Theorem \ref{PDEmthrm}.
\begin{lemma}\label{unifmonotone}
There exists $\beta>0$ such that
$U$ is uniformly monotone in the $e_1$ direction (away from $\partial\Omega$) in the sense of Definition \ref{unifmonotonedefn} with $\alpha=s+1$.
More precisely, there exists $h_0>0$ such that, for any $x_0 \in \Omega$,
\begin{align}
U(x_0)+\beta h^{1+s}\leq U(x_0+e_1h)\qquad \forall \,h<\min\left\{d(x_0,\partial\Omega), h_0\right\}.\notag
\end{align}
\end{lemma}
\begin{proof}
Fix a point $x_0\in \Omega$, and consider a sequence of subsolutions $u^\delta$ such that $U(x_0)\leq u^\delta(x_0)+\delta$, with $\delta \to 0$.
Arguing as in the proof of Proposition \ref{prop:monotone}, up to replace $u^\delta$ with $\sup_{y \in C^+}u^\delta_y$ with $u^\delta_y$
as in \eqref{uhdefn}, we can assume that $u^\delta$ is monotone in all directions $y\in C^+$.
Moreover, let $u$ be a subsolution such that 
\begin{align}
\epsilon d^{s}(x,\partial\Omega^-)\leq u(x)\leq 1-\epsilon d^{s}(x,\partial\Omega^+)\label{growthestimate1}
\end{align}
for some $\epsilon$ small enough, and  $u$ is monotone in all directions $y\in C^+$.
We know such a subsolution $u$ exists because we constructed one in the proof of Lemma \ref{growth2}.
By possibly taking the maximum of $u^\delta$ and $u$ (and applying Lemmas \ref{maxsubissub} and \ref{growth2}) we can assume (\ref{growthestimate1}) holds for $u^\delta$. 

With this in mind we define, for any $\beta>0$ and $h<d(x_0,\partial\Omega)$,
\begin{align}
u^\delta_h(x)= \left\{
\begin{array}{lclcl}
u^\delta(x)&\text{if} &x\in \Omega&\text{and}&d(x,\partial\Omega)\leq 	h,\\
u^\delta(x)\vee (u^\delta(x-he_1)+\beta h^{1+s})&\text{if} &x\in \Omega&\text{and}&d(x,\partial\Omega)>h,\\
1&\text{if} &x\in \Omega^{c,+},\\
0&\text{if} &x\in \Omega^{c,-}.
\end{array}\right.\notag
\end{align}
(See figure \ref{fig6}.)
We will show that there exists a universal $\beta>0$ such that, for $h$ small (the smallness depending only on the geometry of the problem), $u^\delta_h$ is a subsolution. 
Once we have established $u^\delta_h$ is a subsolution it follows that $U(x)\geq u^\delta_h(x)$ and therefore, if $d(x_0,\partial\Omega)>h$,
\begin{align}
U(x_0)\leq u^\delta(x_0)+\delta =u^\delta_h(x_0+he_1)-\beta h^{1+s}+\delta\leq U(x_0+he_1)-\beta h^{1+s}+\delta.\notag
\end{align}
Since $\delta>0$ is arbitrary this will establish the lemma.

Let $\phi\in C^{1,1}(x)\cap BC(\mathbb{R}^N)$ touch $u^\delta_h$ from above at $x$.
If $u^\delta(x)\geq u^\delta(x-he_1)+\beta h^{1+s}$ then $\I\tilde{u}_h^\delta(x)\geq \I\tilde{u}^\delta(x)\geq 0$ where $\tilde{u}_h^\delta$ and
$\tilde{u}^\delta$ are described by (\ref{subsupcut}).  So, it remains to check the case $u^\delta(x-he_1)+\beta h^{1+s}>u^\delta(x)$.

\begin{figure}
\centerline{\epsfysize=2.7truein\epsfbox{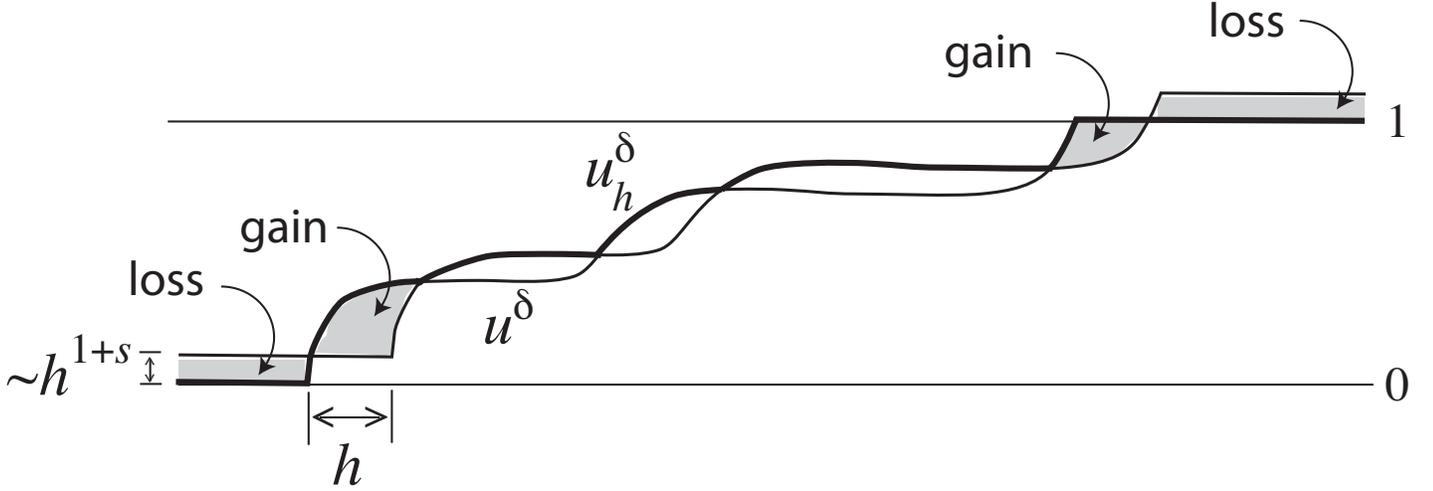}}
\caption{Gain and loss for ${u}^\delta_h$.}
\label{fig6}
\end{figure}

To show $\I\tilde{u}^\delta_h(x)\geq 0$ in this case, we use that $\I(\tilde{u}^\delta(x-he_1)+\beta h^{1+s})\geq 0$,
and then estimate the difference between the operator applied to $u^\delta(\cdot-he_1)+\beta h^{1+s}$ and the operator applied to $u^\delta_h$.
There is a positive contribution from the growth (and decay) near the boundaries (see \eqref{growthestimate1}),
and there is a negative contribution from changing the value of $u^\delta(\cdot-he_1)+\beta h^{1+s}$ inside $\R^N\setminus \Omega$, as depicted in Figure \ref{fig6}.
So, the goal will be to show the gain overpowers the loss for $\beta$ and $h$ small.

Observe that $\phi(\cdot)-\beta h^{1+s}$ touches $u^\delta$ from above at $x-he_1$.
So, by the monotonicity of $u^\delta$ along directions of $C^+$, we know that either $\nabla\phi(x)=0$ or $C_\theta\leq v\cdot e_1$, where $v$ is the direction of $\nabla\phi(x)$.
Consider first the case where $\nabla\phi(x)\neq 0$.  Let $t>0$ denote the distance from $x$ to the set $\{z: d(z,\partial\Omega^-) \leq h\}$ along the line passing through $x$
with direction $\nabla\phi(x)$, and let $t'>0$ denote the distance between $x$ and the set $\{z: d(z,\partial\Omega^+) \leq h\}$ along this same line. Then
\begin{align}
t+h,t'+h<\frac{M}{C_\theta}\label{test}.
\end{align}
Moreover, thanks to the growth estimate \eqref{growthestimate1} on $u^\delta$ near $\partial\Omega$,
we can estimate
\begin{align}
 \I\tilde{u}^\delta_h(x)\geq &\I\tilde{u}^\delta(x-e_1 h)-\int_{-\infty}^{-t-h} \frac{\beta h^{1+s}}{|\eta|^{1+2s}}\,d\eta +\int_{-t-h}^{-t} \frac{[\epsilon(t+h+\eta)^s-\beta h^{1+s}]}{|\eta|^{1+2s}}\,d\eta\notag\\
&+\int_{t'}^{t'+h} \frac{[\epsilon(t'+h-\eta)^s-\beta h^{1+s}]}{|\eta|^{1+2s}}\,d\eta-\int^{\infty}_{t'+h} \frac{\beta h^{1+s}}{|\eta|^{1+2s}}\,d\eta.\notag
\end{align}
The third and fourth terms represent the gain from the growth near the boundary, while in the second and fifth terms
we considered the worst case in which $x+\eta v \in \R^N\setminus \Omega$ for $\eta \in (-\infty,-t-h] \cup [t'+h,\infty)$ so that we have
a loss from the change in $\R^N\setminus \Omega$.
We need to establish
\begin{align}
\frac{1}{2s}\frac{\beta h^{1+s}}{(t+h)^{2s}}=\int_{-\infty}^{-t-h} \frac{\beta h^{1+s}}{|\eta|^{1+2s}}\,d\eta \leq\int_{-t-h}^{-t}
\frac{[\epsilon(t+h+\eta)^s-\beta h^{1+s}]}{|\eta|^{1+2s}}\,d\eta\notag,
\end{align}
and a similar statement replacing $t$ by $t'$.

In that direction, we estimate
\begin{align}
\int_{-t-h}^{-t} \frac{[\epsilon(t+h+\eta)^s-\beta h^{1+s}]}{|\eta|^{1+2s}}\,d\eta&\geq \frac{1}{(t+h)^{1+2s}}\int_{-t-h}^{-t} [\epsilon(t+h+\eta)^s-\beta h^{1+s}]\,d\eta\notag\\
&=\frac{1}{(t+h)^{1+2s}}\left(\frac{\epsilon}{1+s}h^{s+1}-\beta h^{2+s}\right).\notag
\end{align}
The same statement holds replacing $t$ with $t'$. Hence,
\begin{align}
\I\tilde{u}^\delta_h(x_0)\geq &\frac{h^{s+1}}{{(t'+h)}^{2s}}\left(\frac{\epsilon}{1+s}\frac{1}{t'+h} - \frac{\beta}{2s} - \frac{\beta h}{t'+h}\right)\notag\\
&+  \frac{h^{1+s}}{{(t+h)}^{2s}}\left(\frac{\epsilon}{1+s}\frac{1}{t+h} - \frac{\beta}{2s}-\frac{\beta h}{t+h}\right).\notag
\end{align}
Thanks to \eqref{test}, it suffices to choose $\beta=\min\left\{\frac{\epsilon}{2(1+s)},\frac{sC_\theta\epsilon}{(1+s)M}\right\}$
to get $\I\tilde{u}^\delta_h\geq 0$. This concludes the proof when $\nabla\phi(x_0)\neq 0$.  

A nearly identical argument holds for the case $\nabla\phi(x_0)=0$. Indeed, if $y$ is the supremum in the definition of $\I\tilde{u}^\delta_h$ then $C_\theta\leq y\cdot e_1$, a consequence of
the monotonicity of $u^\delta$ in all directions of $C^+$.
Analogously, if $z\in S^{N-1}$ is the direction for the infimum in the definition of $\I\tilde{u}^\delta_h$, then $z \in C^-$.
Let $t'$ is the distance from $x$ to the set $\{z: d(z,\partial\Omega^+) \leq h\}$ along
the line with direction $y$, and let $t$ be the distance from $x$ to $\{z: d(z,\partial\Omega^+) \leq h\}$ along the line with direction $z$.
Then (\ref{test}) holds and we can follow the same argument as above to establish $\I\tilde{u}^\delta_h\geq 0$.
\end{proof}

As we already said before, this uniform monotonicity implies we can only touch $U$ by test
functions which have a non-zero derivative. Indeed the following general lemma holds:
\begin{lemma}\label{notouchlem}
Let $u$ be uniformly monotone in the sense of Definition \ref{unifmonotonedefn} for some $\alpha<2$.  If a test function $\phi\in C^{1,1}(x_0)\cap BC(\mathbb{R}^N)$ touches $u$ from above or from below at $x_0\in \Omega$ then $\nabla\phi(x_0)\neq 0$.
\end{lemma}

\begin{proof}
Say $\phi$ touches $u$ from above at $x_0$ and $\nabla\phi(x_0)=0$.  Applying the uniformly monotone assumption then the $C^{1,1}$ definition we find
\begin{align}
\beta h^{\alpha} \leq u(x_0+hy)-u(x_0) \leq \phi(x_0+hy)-\phi(x_0) \leq M h^2\notag
\end{align}
for some constant $M>0$ and small $h>0$.  As $\alpha <2$, this gives a contradiction for small $h$.  A similar argument is made if $\phi$ touches $U$ from below at $x_0$.
\end{proof}

The above lemma combined with Lemma \ref{notouchlem} immediately gives the following:
\begin{corollary}\label{notouchcor}
If a test function $\phi\in C^{1,1}(x_0)\cap BC(\mathbb{R}^N)$ touches $U$ from above or from below at $x_0\in \Omega$, then $\nabla\phi(x_0)\neq 0$.
\end{corollary}

Thanks to above result, we can take advantage of the fact that the operator $\I$ is stable at non-zero gradient point to show that $U$ is a solution.
We start showing that it is a subsolution.
\begin{lemma}\label{Usubsolnlemma}
$U$ is a subsolution. Moreover $U \in C^{0,2s-1}(\R^n)$.
\end{lemma}
\begin{proof}
Fix $x_0\in\Omega$ and let $\phi\in C^{1,1}(x_0)$ touch $U$ from above at $x_0$. Corollary \ref{notouchcor} implies $\nabla\phi(x_0)\neq 0$.
Setting $m_0:=d(\partial\Omega^-,\partial\Omega^+) \geq C_\theta m>0$ (here $m$ is as in \eqref{eq:m M}), the cusp $1-m_0^{1-2s}|x-x_0|^{2s-1}$ is a member of the family $\mathcal{F}$ defined by
(\ref{monotonefamily}) when $x_0\in\partial\Omega^+$. So, using Lemma \ref{maxsubissub} and Theorem \ref{holdcontthrm} together with the Arzel\`a-Ascoli theorem,
we can easily construct a sequence of equicontinuous and bounded subsolutions $u_n$ such that:
\begin{itemize}
\item $0\leq u_n(x)\leq 1$ for all $x\in \mathbb{R}^N$,
\item $u_n$ converges uniformly to $U$ in compact subsets of $\Omega$.
\end{itemize}
This sequence satisfies the assumptions of Theorem \ref{stability} and we conclude $U$ is a subsolution.
The H\"older regularity of $U$ follows from the one of $u_n$ (equivalently, once we know $U$ is a subsolution, we can deduce its H\"older regularity using Theorem \ref{holdcontthrm}).
\end{proof}

\begin{theorem}\label{monoextthrm}
 $U(x)$ is a solution. Moreover $U \in C_{0,{\rm loc}}^{0,2s-1}(\Omega)$.
\end{theorem}

\begin{proof}
The $C_{0,{\rm loc}}^{0,2s-1}$ regularity of $U$ will follow immediately from Corollary  \ref{cor:improved holder} once we will know that $U$ is a solution.
Let us prove this.

We will assume by contradiction $U$ is not a supersolution, and we will show there exists a subsolution to \eqref{monotonefamily}
which is strictly grater than $U$ at some point, contradicting the definition of $U$.

If $U$ is not a supersolution then there is a point $x_0\in \Omega$, a test function $\psi\in C^{1,1}(x_0)$ touching $U$ from below at $x_0$, and two constants $r,\rho>0$ such that
$\I\hat{U}(x_0)\geq 2\rho$ where
\begin{align}
\hat{U}(x):=\left\{
\begin{array}{ccl}
\psi(x) &\text{if}  &x\in  B_r(x_0),\notag\\
U(x) &\text{if} &x\in \mathbb{R}^N\setminus B_r(x_0).\notag
\end{array}\right.
\end{align}
By possibly considering a smaller $r$ and $\rho$ we may assume $\psi\in C^2(B_r)$ (for instance, it suffices to replace $\psi$ by a paraboloid).
We will show there are $\delta_0,r_0>0$ small such that, for any $\delta\in(0,\delta_0)$, $w_\delta(x)=U(x)\vee [(\psi(x)+\delta)\chi_{B_{r_0}(x_0)}]$ is a subsolution.
(Here and in the sequel, $\chi_S$ denotes the indicator function for the set $S$.)
Clearly $w(x_0)>U(x_0)$, which will contradict the definition of $U(x)$ and prove that $U$ is a supersolution.

For $x\in B_r(x_0)$ we can evaluate $\I\hat{U}(x)$ directly, without appealing to test functions.
First we note $\I\hat{U}(x)$ is continuous near $x_0$, an immediate consequence
of the continuity of $\psi$, $\nabla \psi$, $U$, along with $\nabla\psi(x_0)\neq 0$
given by Corollary \ref{notouchcor}.  Hence, we deduce the existence of a constant $r_0<\frac{r}{2}$ such that $\I\hat{U}(x)>\rho$ and $\nabla\psi(x)\neq 0$ for $x\in B_{r_0}(x_0)$.
Since by assumption $\psi<U$ on $B_r(x_0)\setminus \{x_0\}$, there exists $\delta_1>0$ small enough that $\psi(x)+\delta_1<U(x)$ for all $x\in B_r(x_0)\setminus B_{r_0}(x_0)$.

Consider now $w_\delta$ for any $\delta<\delta_1$.  If $\phi\in C^{1,1}(x)$ touches $w_\delta$ from above at $x\in \Omega$, then it touches either $U$ or
$\psi+\delta$ from above at $x$. In the first case we use the fact that $U$ is a subsolution and calculate
$\I\tilde{w}_\delta(x)\geq I\tilde{U}(x)\geq 0$.
Here $\tilde{w}$ and $\tilde{U}$ are as described in (\ref{subsupcut}).  If $\phi$ touches $\psi$ from above, then $x \in B_{r_0}(x_0)$ and $\nabla \phi(x)=\nabla\psi(x)\neq 0$. So,
since $\psi(x)+\delta_1<U(x)$ for all $x\in B_r(x_0)\setminus B_{r_0}(x_0)$ and $r-r_0 \geq r_0$, we get
\begin{align}
\I\tilde{w}_\delta(x)&\geq \I\hat{U}(x)-2\delta\int_{r-r_0}^\infty\frac{1}{\eta^{1+2s}}\,ds \geq  \rho-2\delta\int_{r_0}^\infty\frac{1}{\eta^{1+2s}}\,d\eta =
\rho - \frac{\delta r_0^{2s}}{s}.\notag
\end{align}
Hence it suffices to choose $\delta_0=s\rho r_0^{2s}$ to deduce that 
$\I\tilde{w}_\delta\geq 0$ and find the desired contradiction.
\end{proof}

\begin{remark}{\rm
Using Corollary \ref{cor:improved holder} we obtained that  $U \in C_{0,{\rm loc}}^{0,2s-1}(\Omega)$.
However, one can actually show that $U \in C_{0}^{0,2s-1}(\R^N)$. To prove this, one uses a blow-up argument as in the proof of Corollary \ref{cor:improved holder}
together with the fact that $U$ grows at most like $d^s(\cdot,\partial \Omega) \ll d^{2s-1}(\cdot,\partial \Omega)$ near $\partial\Omega$
(see Remark \ref{rmk:growth}) to show that any blow-up profile $u_0$
solves $\I u_0=0$ inside some infinite domain $\tilde \Omega \subset \R^N$, and vanishes outside. Then, arguing as the proof of Lemma \ref{impreg1} one obtains $u_0=0$.
We leave the details to the interested reader.
}
\end{remark}

We finally establish uniqueness of solutions by proving a general comparison principle which does not rely on compactness of $\Omega$
but uses the stability of $\I$ when the limit function cannot be touched by a test function with zero derivative.
\begin{theorem}\label{monocompthrm}
Let $\Omega$ be bounded in the $e_1$ direction (i.e., $\Omega\subset \{-M \leq x_1 \leq M\}$ for some $M>0$). Consider two functions $u,w:\mathbb{R}^N\rightarrow \mathbb{R}$ such that
\begin{itemize}
\item $\I u\geq 0$ and $\I w\leq 0$ ``at non-zero gradient points'' inside $\Omega$ (in the sense of Definition \ref{subsupdefn}),
\item $u\leq w$ on $\mathbb{R}^N\setminus \Omega$,
\item $u,w \in C^{0,2s-1}(\R^N)$,
\item $u$ or $w$ is uniformly monotone along $e_1$ away from $\partial\Omega$ (see Lemma \ref{unifmonotone}) for some $\alpha<2$.
\end{itemize}
Then $u\leq w$ in $\Omega$.
\end{theorem}
\begin{proof}
By way of contradiction, we assume there is a point $x\in \Omega$ such that $u(x)> w(x)$.
Replacing $u$ and $w$ by $u^\epsilon$ and $w_\epsilon$, 
we have  $u^\epsilon(x)-w_\epsilon(x) \geq c>0$ for $\epsilon$ sufficiently small.
Moreover, since $u,w \in C^{0,2s-1}(\R^N)$, the uniform continuity of $u$ and $w$, together with the assumption $u\leq w$ on $\mathbb{R}^N\setminus \Omega$, implies that
$(u^\epsilon-w_\epsilon)\vee 0 \to 0$ as $\epsilon \to 0$ uniformly on $\R^N\setminus \Omega$ (see Lemma \ref{infsupconvlemma}).
So, assume that $\epsilon$ is small enough that $u^\epsilon-w_\epsilon \leq c/2$ on $\R^N\setminus \Omega$, and define
\begin{align*}
\delta_0&=\inf\{\delta: w_\epsilon(x)+\delta\geq u^\epsilon(x) \ \ \forall \,x\in \mathbb{R}^N\}\\
&=\inf\{\delta: w_\epsilon(x)+\delta\geq u^\epsilon(x) \ \ \forall \,x\in \Omega\}.
\end{align*}
Observe that $\delta_0\geq c>0$.

If there is a point $x_0\in \Omega$ such that $w_\epsilon(x_0)+\delta_0=u^\epsilon(x_0)$ we can proceed as in the proof of the comparison principle for compact sets
(Theorem \ref{comparison}) to obtain a contradiction.
If no such $x_0$ exists, let $x_n\in \Omega$ be a sequence such that $w_\epsilon(x_n)+\delta_0\leq u^\epsilon(x_n)+\frac{1}{n}$.
Let us observe that since  $u\leq w$ in $\mathbb{R}^N\setminus \Omega$, $u,w \in C^{0,2s-1}(\R^N)$, and $u^\epsilon -u,w_\epsilon -w$ are uniformly close to zero,
the points $x_n$ stay at a uniform positive distance from $\partial\Omega$ for $\epsilon$ sufficiently small .
Let $\tau_n\in \{0\}\times\mathbb{R}^{N-1}$ be such that $\tilde{x}_n:=x_n-\tau_n=(e_1\cdot x_n)e_1$, that is $\tau_n$
is the translation for which $\tilde{x}_n\in \mathbb{R}\times \{(0,...,0)\}$.
The sequence $\tilde{x}_n$ lies in a bounded set of $\mathbb{R}^N$ since $\Omega$ is bounded in the $e_1$ direction,
so we may extract a subsequence with a limit $x_0$.
Being (sub/super)solution invariant under translations, $u_n(x)=u^\epsilon(x-\tau_n)$ and $w_n(x)=w_\epsilon(x-\tau_n)$
form a family of ``sub and supersolutions at non-zero gradient points'' which are uniformly equicontinuous and bounded.
Using the Arzel\`a-Ascoli theorem, up to a subsequence we can find two functions $u_0(x)$ and $w_0(x)$ such that $u_n(x)\rightarrow u_0(x)$ and $w_n(x)\rightarrow w_0(x)$
uniformly on compact sets. By the stability Theorem \ref{stability}, $u_0$ [resp. $w_0$] is a ``subsolution [resp. supersolution] at non-zero gradient points''.
Moreover $u_0(x)\leq w_0(x)+\delta_0$ for all $x\in\mathbb{R}^N$, $u_0(x_0)=w_0(x_0)+\delta_0$, and 
\begin{align}
w_0(x)-u_0(x)\geq \delta_0 \qquad \text{on }\mathbb{R}^N\setminus \Omega.\label{wgubdry}
\end{align}
Furthermore, the uniform $C^{1,1}$ bounds from below [resp. above] on $u^\epsilon$ [resp. $w_\epsilon$] (see Lemma \ref{infsupconvlemma})
implies that also $u_0$ [resp. $w_0$] is $C^{1,1}$ from below [resp. above].
Finally, if we assume for instance that $u$ is uniformly monotone along $e_1$ away from $\partial\Omega$ (see Lemma \ref{unifmonotone}) for some $\alpha<2$, then also
$u_0$ is uniformly monotone along $e_1$ away from $\partial\Omega$ for the same value of $\alpha$.

Now, to find the desired contradiction, we can argue as in the proof of Theorem \ref{comparison}: at the point $x_0$ both functions are $C^{1,1}$ at $x_0$,
so Lemma \ref{notouchlem} applied with $u=u_0$ and $\phi=w_0$ implies $\nabla w_0(x_0)\neq 0$,
and the sub and supersolution conditions at $x_0$ give a contradiction. This concludes the proof.
\end{proof}

\section{A Monotone Obstacle Problem}\label{obstacle}
In this section we consider the problem of finding a solution $u:\mathbb{R}^N\rightarrow\mathbb{R}$ of the following dual obstacle problem:
\begin{align}
\left\{\begin{array}{cclcl}
\I u(x)&\geq&0 &\text{if}  &\Gamma^-(x)<u(x),\\
\I u(x)&\leq&0 &\text{if}	  &\Gamma^+(x)>u(x),\\
u(x)&\geq&\Gamma^-(x) &\text{for all}  &x\in \mathbb{R}^N,\\
u(x)&\leq&\Gamma^+(x) &\text{for all}  &x\in \mathbb{R}^N.
\end{array}\right.\label{obstaclePDE}
\end{align}
Here $\Gamma^+$ and $\Gamma^-$ are upper and lower obstacles which confine the solution and we interpret the above definition in the viscosity sense.
The model case one should have in mind is $\Gamma^+=[(x_1)_-]^{-\gamma_1}\wedge 1$ and $\Gamma^-=(1- [(x_1)_+]^{-\gamma_2}) \vee 0$, $\gamma_1,\gamma_2 >0$.
(Here and in the sequel, $x_1$ denotes the component of $x$ in the $e_1$ direction.)

\begin{figure}
\centerline{\epsfysize=2truein\epsfbox{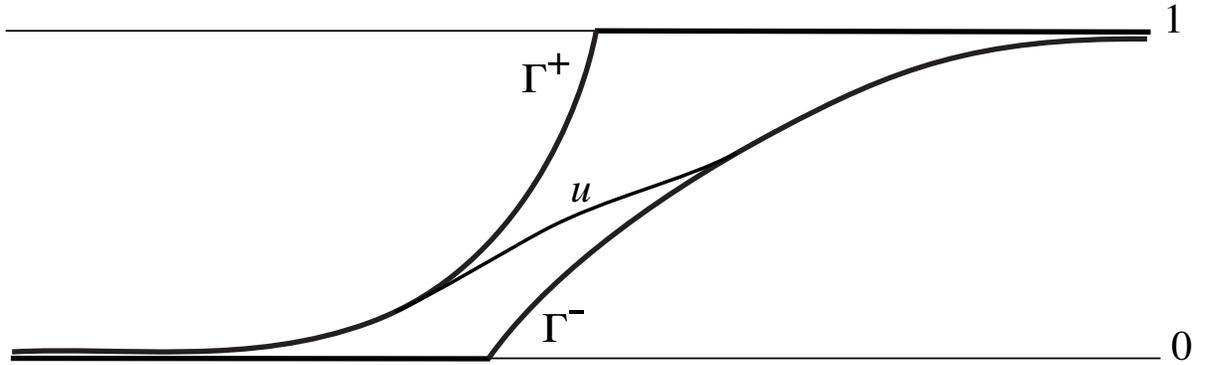}}
\caption{The dual obstacle problem.}
\label{fig7}
\end{figure}

When $u$ does not coincide with $\Gamma^-$ we require it to be a subsolution, and likewise it must be a supersolution when it does not coincide with $\Gamma^+$.
In particular, when it does not coincide with either obstacle it must satisfy $\I u(x)=0$.

We make the following assumptions:
\begin{itemize}
\item $ 0\leq \Gamma^- <\Gamma^+\leq 1$. 
\item $\Gamma^+,\Gamma^-$ are uniformly Lipschitz on $\R^N$ (and we denote by $L_0$ be a Lipschitz constant for both $\Gamma^+$ and $\Gamma^-$).
\item $\Gamma^+$ [resp. $\Gamma^-$] is uniformly $C^{1,1}$ inside the set $\{ \Gamma^+ < 1\}$ [resp. $\{\Gamma^->0\}$].
\item $\Gamma^+(x)\rightarrow 0$ uniformly as $x_1\rightarrow -\infty$ and $\Gamma^-(x)\rightarrow 1$ uniformly as $x_1\rightarrow\infty$. More precisely,
\begin{align}
\lim_{x_1\rightarrow -\infty} \sup_{\bar{x}\in\mathbb{R}^{N-1}} \Gamma^+(x_1,\bar{x}) =0, \label{limasspos}\\
\lim_{x_1\rightarrow +\infty} \sup_{\bar{x}\in\mathbb{R}^{N-1}} \Gamma^-(x_1,\bar{x}) =1. \label{limassneg}
\end{align}
\item $\Gamma^+,\Gamma^-$ are monotone in all directions for some cone of monotonicity $C^+$ with $\theta\in (0,\pi/2)$ (see \eqref{monocone}).
Moreover, the monotonicity is strict away from $0$ and $1$:
for every $M>0$ there exists $l_M>0$ such that
\begin{equation}
\label{eq:bound from below}
\begin{split}
 l_M \leq& \frac{|\Gamma^+(x)-\Gamma^+(y)|}{|x-y|} \qquad \forall\,x,y \in \{ \Gamma^+ < 1\}\cap \{|x_1|\leq M\},\, x-y \in C^+;\\
l_M \leq& \frac{|\Gamma^-(x)-\Gamma^-(y)|}{|x-y|} \qquad \forall\,x,y \in \{ \Gamma^- > 0\}\cap \{|x_1|\leq M\}\, x-y \in C^+.
\end{split}
\end{equation}
Again, we write $C_\theta=\cos(\pi/2-\theta)$. 
\item For each $M>0$ there is a constant $L_M$ such that if $x,y\in \{x:M<|x_1|<2M\}$ then
\begin{align}
\frac{|\Gamma^+(x)-\Gamma^+(y)|}{|x-y|}+\frac{|\Gamma^-(x)-\Gamma^-(y)|}{|x-y|} \leq L_M,\label{lipasspos}
\end{align}
and $L_M\rightarrow 0$ as $M\rightarrow \infty$.
\item There exist $\alpha>2s$ and global constants $M_0>0$, $\rho_0\in (0,1)$ with the following property:  For every $\tilde{x}$ with $\tilde{x}_1>M_0$ there exists $A>0$
(which may depend on $\tilde x$) satisfying 
\begin{align}
C|\tilde{x}_1|^{-\alpha}&\geq A^{s}\left(1-\Gamma^-(\tilde{x})+\frac{|\nabla\Gamma^-(\tilde{x})|^2}{4A}\right)^{1-s},\label{parabassnega}\\
\rho_0&> \left(1-\Gamma^-(\tilde{x})+\frac{|\nabla\Gamma^-(\tilde{x})|^2}{4A}\right),\label{parabassnegb}
\end{align}
such that $\Gamma^-$ can be touched from above by the paraboloid
\begin{align}
p_{\tilde{x}}^-(x)=\Gamma^-(\tilde{x})+\nabla\Gamma^-(\tilde{x})\cdot(x-\tilde{x})+A(x-\tilde{x})^2.\label{parabneg}
\end{align}
For every $\tilde{x}$ with $\tilde{x}_1>M$ there exists $A>0$ satisfying 
\begin{align}
C|\tilde{x}_1|^{-\alpha}&\geq A^{s}\left(\Gamma^+(\tilde{x})+\frac{|\nabla\Gamma^+(\tilde{x})|^2}{4A}\right)^{1-s},\label{parabassposa}\\
\rho_0&> \left(\Gamma^+(\tilde{x})+\frac{|\nabla\Gamma^+(\tilde{x})|^2}{4A}\right),\label{parabassposb}
\end{align}
such that $\Gamma^+$ can be touched from below by the paraboloid
\begin{align}
p_{\tilde{x}}^+(x)=\Gamma^+(\tilde{x})+\Gamma^+(\tilde{x})\cdot(x-\tilde{x})-A(x-\tilde{x})^2.\label{parabpos}
\end{align}
\end{itemize}
Assumption \eqref{eq:bound from below} is used to establish uniform monotonicity of the solutions.
Assumptions (\ref{parabassnega})-(\ref{parabpos}) control the asymptotic behavior of the obstacles and
guarantee the solution coincides with the obstacle in some neighborhood of infinity.
It is important to note that a different $A$ maybe chosen for each $\tilde x$.  These assumptions are realized in the two following general situations:

\begin{itemize}
\item (Polynomial Control of the Obstacles) Let $\gamma_1,\gamma_2>0$, and $\Gamma^+,\Gamma^-$ be such that the following holds: There are global constants $C,M>0$ such that
\begin{itemize}
\item[i.] $\Gamma^+(x)\leq C|x_1|^{-\gamma_1}$ when $x_1<-M$, and $1-\Gamma^-(x)\leq C|x_1|^{-\gamma_2}$ when $x_1>M$.
\item[ii.] $\frac{1}{C}|x_1+1|^{-\gamma_1-1}\leq |\nabla \Gamma^+(x)|\leq C|x_1+1|^{-\gamma_1-1}$  for all $x\in\mathbb{R}^N$.
\item[iii.] $\frac{1}{C}|x_1+1|^{-\gamma_2-1}\leq |\nabla \Gamma^-(x)|\leq C|x_1+1|^{-\gamma_2-1}$  for all $x\in\mathbb{R}^N$.
\item[iv.] The paraboloids (\ref{parabneg}) and (\ref{parabpos}) satisfy  $\frac{1}{C}|x_1|^{-\gamma_1-2}\leq A \leq C|x_1|^{-\gamma_1-2}$ if $x_1<-M$
and   $\frac{1}{C}|x_1|^{-\gamma_2-2}\leq A \leq C|x_1|^{-\gamma_2-2}$ if $x_1>M$, respectively.
\end{itemize}
Then $\Gamma^+$ and $\Gamma^-$ satisfy \eqref{eq:bound from below}-\eqref{parabpos}.
\item If we choose $A\sim |\nabla\Gamma^+|^2$ for (\ref{parabpos}) and $A\sim |\nabla\Gamma^-|^2$ for (\ref{parabneg}) then 
\begin{equation}
\label{convass}
\begin{split}
&\frac{1}{C}|x_1+1|^{-\gamma_1-1}\leq |\nabla \Gamma^+(x)|\leq C|x_1+1|^{-\gamma_1-1},\\
&\frac{1}{C}|x_1+1|^{-\gamma_2-1}\leq |\nabla \Gamma^-(x)|\leq C|x_1+1|^{-\gamma_2-1},
\end{split}
\end{equation}
implies (\ref{parabassnega})-(\ref{parabassnegb}) and (\ref{parabassposa})-(\ref{parabassposb}) for any $\gamma_1,\gamma_2>0$.
In particular, if $\Gamma^+$ and $\Gamma^-$ are concave and convex respectively in a neighborhood of infinity on the $x_1$ axis, then the paraboloids
(\ref{parabneg}) and (\ref{parabpos}) touch for any $A\geq0$, and (\ref{convass}) is enough to ensure \eqref{eq:bound from below}-\eqref{parabpos}.
\end{itemize}
 
To define the family $\mathcal{F}$ of admissible subsolutions, we say $u\in \mathcal{F}$ if it satisfies
\begin{align}
\left\{\begin{array}{cclcl}
\I u(x)&\geq&0 &\text{if}  &\Gamma^-(x)<u(x),\\
u(x)&\geq&\Gamma^- &\text{for all}  &x\in \mathbb{R}^N,\\
u(x)&\leq&\Gamma^+ &\text{for all}  &x\in \mathbb{R}^N.
\end{array}\right.\label{obstaclefamily}
\end{align}
This set in non-empty because it contains $\Gamma^-(x)$.
Again we will use Perron's method to show the supremum of functions in this set solves the problem (\ref{obstaclePDE}).
Our solution candidate is:
\begin{align}
U(x)=\sup\{u(x): u\in\mathcal{F}\}.\notag
\end{align}

We prove the following existence and uniqueness result:
\begin{theorem}\label{obstaclePDEthrm}
Let $\Gamma^+$ and $\Gamma^-$ satisfy the stated assumptions, then $U(x)$ is the unique solution to the problem (\ref{obstaclePDE}).  
\end{theorem}

Arguing the existence and uniqueness given by this theorem is similar to the work in the previous section.
We will build barriers from paraboloids which will show $U$ coincides with the obstacle near infinity along the $e_1$ axis.
From here, we will make use of the structure of the obstacles to show $U$ is uniformly monotone in the sense of Definition \ref{unifmonotonedefn}
so that it may only be touched by test functions with non-zero derivatives (actually, we will show $U$ has locally a uniform linear growth).
Again, this implies stability, allowing us show that $U$ is the (unique) solution.

We also show the solution is Lipschitz, and demonstrate that $U$ approaches the obstacle in a $C^{1,s}$ fashion along the direction of the gradient.
This is the content of the following theorem. Recall that  $L_0$ denotes a Lipschitz constant for both $\Gamma^+$ and $\Gamma^-$. 
\begin{theorem}\label{obstacleregthrm}
There exists a constant $A_\theta\geq 1$, depending only on the opening of the cone $C^+$, such that
$U$ is $(A_\theta L_0)$-Lipschitz. 
Furthermore, $U$ approaches the obstacle in a $C^{1,s-1/2}$ fashion along the direction of the gradient:
If $x$ is such that $\Gamma^+(x)=U(x)$ and $y\in S^{N-1}$ is the direction of $\nabla \Gamma^+(x)$, then
\begin{align}
\Gamma^+(x+r y)-U(x+r y) =O(r^{1+(s-1/2)})\notag
\end{align}
as $r\rightarrow 0$. (Here, the right hand side $O(r^{1+(s-1/2)})$ is uniform with respect to the point $x$.)   Similarly,
if $x$ is such that $\Gamma^-(x)=U(x)$ and $y\in S^{N-1}$ is the direction of $\nabla \Gamma^-(x)$, then
\begin{align}
U(x-r y)-\Gamma^-(x-r y) =O(r^{1+(s-1/2)})\notag
\end{align}
as $r\rightarrow 0$.\end{theorem}

\subsection{Barriers and Uniform Monotonicity}
One may show $U$ is monotone in the directions of $C^+$ by arguing similar to Section \ref{monosec}.
Indeed, given $u \in \mathcal F$ and a direction $y \in C^+$, one can easily show that $u(\cdot+y) \vee \Gamma^-$ still belongs to $\mathcal F$.

Our first goal is to demonstrate how the solution coincides with the obstacle in a neighborhood of infinity.
 We begin with a lemma similar in spirit to Lemma \ref{growth1construction}, constructing barriers with paraboloids.\
Recall that $C^-=-C^+$.

\begin{lemma}\label{obsbarriers}
Let $p^+$ and $p^-$ be given by (\ref{parabneg}) and (\ref{parabpos}) respectively.    For any $\mathcal{S}\subset\mathbb{R}^N$ define
\begin{align}
P^+(x)&=P^+_{\mathcal{S}}(x)\vee \Gamma^-(x),\label{obsbarriersup}\\
P^+_{\mathcal{S}}(x)&=\sup_{x_0\in \mathcal{S}}\left(\sup_{x' - x_0\in C^+}\left\{p_{x_0}^+(x')\vee 0\right\}\right),\notag
\end{align}
and
\begin{align}
P^-(x)&=P^-_{\mathcal{S}}(x)\wedge \Gamma^+(x),\label{obsbarriersub}\\
P^-_{\mathcal{S}}(x)&=\inf_{x_0\in \mathcal{S}}\left(\inf_{x'-x_0\in C^-}\left\{p_{x_0}^-(x')\wedge 1\right\}\right).\notag
\end{align}
Then there exists $\tilde{M}>0$ such that if $\mathcal{S}\subset \{x:x_1<-\tilde{M}\}$ then $P^+  \in \mathcal{F}$, and if $\mathcal{S}\subset \{x:x_1>\tilde{M}\}$ then $u\leq P^-$ for any $u\in \mathcal{F}$.
\end{lemma}
\begin{proof}
We will first prove $P^+(x)\in\mathcal{F}$ when $\mathcal{S}\subset \{x:x_1<-\tilde{M}\}$. Let $\rho_0\in (0,1)$ be given by the left hand side of (\ref{parabassposb}),
and choose $M$ so that $\Gamma^-(x)\geq \frac{1+\rho_0}{2}$ when $x_1>M$.  The existence of such $M$ is guaranteed by (\ref{limasspos}).
Instead of working directly with $P^+$ we will instead show that
$$
P(x)=P^+_{\mathcal{S}}(x)\vee \frac{(1-\rho_0)}{2}\chi_{\{x_1>M\}}(x)
$$
is a subsolution for any $x$ such that $x_1<M$. This will imply $P^+\in \mathcal{F}$ since $P^+=P$ whenever $P^+$ does note coincide with $\Gamma^-$.

To begin we rewrite 
\begin{align}
p^+(x)=\Gamma^+(\tilde{x})+\frac{|\nabla\Gamma^-(\tilde{x})|^2}{4A}-A\left(x-\tilde{x}-\frac{\nabla\Gamma^-(\tilde{x})}{2A}\right)^2.\notag
\end{align}
From here, we use Lemma \ref{lemma:I parab} with
\begin{align}
r_0^2=r_0(\tilde x)^2=\frac{1}{A}\left(\Gamma^-(\tilde{x})+\frac{|\nabla\Gamma^-(\tilde{x})|^2}{4A}\right),\notag
\end{align}
and arguing as in the proof of Lemma \ref{growth1construction} yields
\begin{align}
\I P(x)\geq \frac{\frac{1+\rho_0}{2}-Ar_0^2}{2s}\left(\frac{C_\theta}{M-\tilde{x}_1}\right)^{2s}-C_s Ar_0^{2-2s}.\notag
\end{align}
By assumption $Ar_0^2 \leq \rho<1$ and $Ar_0^{2-2s}\leq C|\tilde x_1|^{-\alpha}$ for some $\alpha>2s$ (thanks to 
(\ref{parabassposa})). Hence, there exists a large constant $\tilde{M}>0$ such that if $\tilde{x}_1<-\tilde{M}$ then $\I P(x)\geq 0$.
This implies that $P^+\in\mathcal{F}$ if $\mathcal{S}\subset \{x:x_1<-\tilde{M}\}$, as desired.

Similarly one argues $\I P^-(x)\leq 0$ when $\mathcal{S}\subset \{x:x_1>\tilde{M}\}$.

The final statement in the lemma is now an application of the comparison principle on compact sets (Theorem \ref{comparison}): for any point $\bar x \in \{x:x_1>\tilde{M}\}$,
we choose $\mathcal S=\{\bar{x} \}$. Then, there exists a large ball $B_R(\bar x)$ such that $\Gamma^+ \leq P^-$ outside this ball. 
Since any function $u \in \mathcal F$ lies below $\Gamma^+$, we conclude using Theorem \ref{comparison}.
\end{proof}

\begin{corollary}
\label{cor:touch barriers}
Let $\tilde{M}>0$ be given by the previous lemma. Then $U(x)=\Gamma^-(x)$ when $x_1>\tilde{M}$, and $U(x)=\Gamma^+(x)$ when $x_1<-\tilde{M}$.
\end{corollary}
\begin{proof}
If $\tilde{x}$ is such that $\tilde{x}_1<-\tilde{M}$, choose $\mathcal{S}=\{\tilde{x}\}$ and consider $P^+$ as in the previous lemma.
Since $U$ is the maximal subsolution, it must be greater than or equal to $P^+(x)$, In particular $U(\tilde{x})\geq P^+(\tilde{x})=\Gamma^+(\tilde{x})$,
from which we conclude $U(\tilde{x})=\Gamma^+(\tilde{x})$ (since $U\leq \Gamma^+$).

One argues similarly when $\tilde{x}_1>\tilde{M}$, but using the comparison principle on compact sets as at the end of the proof of the previous lemma.
We leave the details to the reader.
\end{proof}

We will now use the fact that $U$ coincides with the obstacles for large $|x_1|$ to show that $U$ is uniformly monotone (compare with Lemma \ref{unifmonotone}). 
In fact we can do a little better in this situation, and show that the growth is linear.

The strategy of the proof is analogous to the one of Lemma \ref{unifmonotone}: given $u\in \mathcal{F}$ that coincides with the obstacles in a neighborhood of infinity,
we compare it with $u(x-he_1)+\beta h$ for some small $\beta>0$.  When we modify $u(x-he_1)+\beta h$ to take into account the obstacle conditions
there will be a loss coming from changing the function near infinity, and a gain coming from changing the function near the obstacle contact point
(since that point is far from infinity, assumption \eqref{eq:bound from below} implies that the gradient of the obstacle bounded uniformly away from zero).
The goal will be to show that the gain dominates the loss  when $\beta$ is small, so the shifted function is a subsolution as well. 

To estimate the loss, given $h>0$ and $y \in S^{N-1}\cap C^+$, we consider the sets 
\begin{align}
\Sigma^+_{h,y}=\{x: \Gamma^+(x-yh)+\beta h \geq \Gamma^+(x)\},\notag\\
\Sigma^-_{h,y}=\{x: \Gamma^-(x-yh)+\beta h \geq \Gamma^-(x)\}.\notag
\end{align}
We have the following lemma:

\begin{lemma}\label{obssetlem}
For every $\beta>0$ there exist $M_\beta>0$ such that
\begin{align}
\Sigma^+_{h,y}\cap\Sigma^-_{h,y}\subset \left\{x: |x_1|\geq M_\beta\right\} \qquad \forall\,h \in (0,1).\notag
\end{align}
Moreover $M_\beta \to \infty$ as $\beta \to 0$.
\end{lemma}
\begin{proof}
Let $y\in S^{N-1}$ be any direction in the assumed cone of monotonicity.  Consider first the obstacle $\Gamma^+$,
and for any $M>0$ let $L_M$ be given by (\ref{lipasspos}).  Then, if $x$ is such that $x_1<-M$ we have
\begin{align}
\Gamma^+(x-hy)+\beta h-\Gamma^+(x)&\geq \beta h-L_{M}h\notag
\end{align}
so that if $L_{M}\leq \beta$ the inclusion for $\Sigma^+_{h,y}$ follows.  For the obstacle $\Gamma^-$, if $x_1>2M$ and $h<M$ then
\begin{align}
\Gamma^-(x-hy)+\beta h-\Gamma^-(x)&\geq \beta h-L_{2M}h\notag
\end{align}
so that if $L_{2M}\leq \beta$ we have the inclusion for $\Sigma^-_{h,y}$.
Since $L_M\rightarrow 0$ as $M\rightarrow \infty$, for any fixed $\beta$ we may take $M$ large enough so that $L_M,L_{2M}\leq \beta$. Then the conclusion holds with $M_\beta=2M$.
\end{proof}

\begin{lemma}\label{unifmonotone2}
Let $\tilde{M}$ be given by Lemma \ref{obsbarriers}. Then $U$ is uniformly monotone in any direction of $C^+$ inside $\{|x_1|\leq \tilde{M}\}$. More precisely, there is a $\beta>0$ such that,
for any $x \in \{|x_1|\leq \tilde{M}\}$, there exists $h_x>0$ such that 
\begin{align}
U(x)+\beta h\leq U(x+yh)\qquad \forall \,h \in (0,h_x),\, y\in S^{N-1}\cap C^+.\notag
\end{align}
Hence, in the sense of distributions,
$$
D_y U(x)\geq \beta>0 \qquad \forall\,x \in \{|x_1|\leq \tilde{M}\},\, y\in S^{N-1}\cap C^+.
$$
\end{lemma}
\begin{proof}
Thanks to assumption \eqref{eq:bound from below} it suffices to prove the result when $\Gamma^- < U < \Gamma^+$.

The proof is similar in spirit to the one of Lemma \ref{unifmonotone}.
Fix a point $x$ such that $\Gamma^-(x)<U(x)<\Gamma^+(x)$, and for any $\delta>0$ let $u^\delta\in\mathcal{F}$ be such that $U(x)\leq u^\delta(x)+\delta$.
Using Lemmas \ref{obsbarriers} and \ref{maxsubissub} (see also Corollary \ref{cor:touch barriers})
and by considering the maximum of two subsolutions, we may assume that $u^\delta$ coincides with $\Gamma^+$ for all $x$ such that $x_1<-\tilde{M}$.
Moreover, since $u^\delta \leq U$, by Corollary \ref{cor:touch barriers} it coincides with $\Gamma^-$ for all $x$ such that $x_1>\tilde{M}$.
Furthermore, as in the proof of Lemma \ref{unifmonotone}, up to replacing $u^\delta$ by $\sup_{y \in C^+} u^\delta (\cdot -y)\vee \Gamma^-$, we can assume 
that $u^\delta$ is monotone in all directions of $C^+$.
For any $\beta>0$, let $M_\beta$ be given by Lemma \ref{obssetlem}. We assume $\beta$ is sufficiently small so that the inclusion in Lemma \ref{obssetlem} holds with $M_\beta>\tilde{M}+1$.  

For $0<h<1$ and $y\in S^{N-1}$ in the assumed cone of monotonicity, consider
\begin{align}
u^\delta_h(x)=\left\{
\begin{array}{ll}
\Gamma^-(x) &\text{if } x\in \{x:x_1>M_\beta\},\\
\Gamma^+(x) &\text{if } x\in \{x:x_1<-M_\beta\},\\
u^\delta(x)\vee (u^\delta(x-yh)+\beta h) &\text{otherwise}.
\end{array}\right.\notag
\end{align}
This is the maximum of $u^\delta$ and ``$u^\delta$ shifted and raised'', then pushed above $\Gamma^-$ and below $\Gamma^+$.
We wish to show $u^\delta_h\in \mathcal{F}$ so we need to check $\I u^\delta_h(x)\geq 0$ when $u^\delta_h(x)\neq \Gamma^-(x)$.
Fix $x\in \mathbb{R}^N$ such that $\Gamma^-(x)<u^\delta_h(x)$ and let $\phi\in C^{1,1}(x)$ touch $u^\delta_h$ from above at $x$.
If $\phi$ touches $u^\delta$ from above, we use that $u^\delta$ is a subsolution and $u^\delta_h\geq u$ to get
$\I\tilde{u}^\delta_h(x)\geq \I\tilde{u}^\delta(x)\geq 0$ (here $\tilde{u}^\delta_h$ is the usual modification given by (\ref{subsupcut})).
In the other case, we assume $\nabla\phi(x)\neq 0$ (as indicated in the proof of Lemma \ref{unifmonotone} the case $\nabla\phi(x)=0$ is handled in a nearly identical way).
Then, 
\begin{align}
\I\tilde{u}^\delta_h(x)\geq \I\tilde{u}^\delta&(x-hy)-\int_{x+v\eta\in \Sigma^+_{h,y}\cap\Sigma^-_{h,y}} \frac{\beta h}{\eta^{1+2s}}\,d\eta\notag\\
&+\int_{x+v\eta\in \Lambda} \frac{u^\delta(x+v\eta)-u^\delta(x+v\eta-h)-\beta h}{\eta^{1+2s}}\,d\eta.\notag
\end{align}
Here $v\in S^{N-1}$ is the direction of $\nabla\phi(x)$ (which satisfies $v\cdot e_1 \geq C_\theta$, thanks to the monotonicity of $u^\delta$),
and $\Lambda$ is the set of values where $u^\delta(\cdot)\geq u^\delta(\cdot-h)-\beta h$.  

First, $\I\tilde{u}^\delta(x-hy)\geq 0$ since $u^\delta$ is a subsolution. So, we need to work on the remaining terms.
To estimate the loss term (second term on the right hand side) we use Lemma \ref{obssetlem} to see 
\begin{align}
\int_{x+v\eta\in \Sigma^+_{h,y}\cap\Sigma^-_{h,y}} \frac{\beta h}{\eta^{1+2s}}\,d\eta&\leq 2\beta h \int^\infty_{|x_1+M_\beta|/C_\theta}\eta^{-2s-1}\,d\eta
=\frac{C_\theta^{2s}\beta h |x_1+M_\beta|^{-2s}}{s}.\notag
\end{align}
To estimate the gain term (third term on right hand side), we note that, if $\beta \leq l_{\tilde{M}}$ and $x+v\eta \in \{ -\tilde{M}-1 \leq x_1 \leq -\tilde{M}\}$, then
\begin{align*}
u^\delta(x+v\eta)-u^\delta(x+v\eta-h)-\beta h &= \Gamma^+(x+v\eta)-\Gamma^+(x+v\eta-h)-\beta h\\
& \geq (l_{\tilde{M}} - \beta) h >0.
\end{align*}
Hence $\{ -\tilde{M}-1 \leq x_1 \leq -\tilde{M}\} \subset \Lambda$, and we get
\begin{align}
\int_{x+v\eta\in \Lambda} &\frac{u^\delta(x+v\eta)-u^\delta(x+v\eta-h)-\beta h}{\eta^{1+2s}}\,d\eta\notag\\
&\geq\int_{x+v\eta\in \{ -\tilde{M}-1 \leq x_1 \leq -\tilde{M}\}} \frac{(l_{\tilde{M}}-\beta) h}{\eta^{1+2s}}\,d\eta\notag\\
&\geq (l_{\tilde{M}}-\beta) h\left(\frac{C_\theta}{|x_1|+\tilde{M}+1}\right)^{1+2s}.\notag
\end{align}
Since $|x_1|<M$, it is clear that $\I\tilde{u}^\delta_h(x)\geq 0$ for $\beta$ sufficiently small.
\end{proof}


\subsection{Existence and Regularity}
This subsection contains the proofs of Theorems \ref{obstaclePDEthrm} and \ref{obstacleregthrm}.

\begin{proof}[Proof of Theorem \ref{obstaclePDEthrm}]
The first step is to check $U\in \mathcal{F}$.  Since $U$ may only be touched from below by a test function with non-zero derivative (as a consequence of Lemma \ref{unifmonotone2})
this is an immediate consequence of the stability Theorem \ref{stability}, arguing as in the proof of Theorem \ref{Usubsolnlemma}.

Next one checks $\I U(x)\leq 0$ whenever $U(x)<\Gamma^+(x)$.  This is argued similar to the proof of Theorem \ref{monoextthrm}: arguing by contradiction,
one touches $U$ from below at any point where it is not a supersolution, then raising the test function one finds another member of the set $\mathcal{F}$ which is strictly greater
than $U$ at some point.
Thus $U$ is a solution to (\ref{obstaclePDE}).

To prove the uniqueness of the solutions we first notice that, if $W$ is any other solution, then $W$ must also coincide with the obstacles for all $x$ with $|x_1|>\tilde{M}$,
where $\tilde{M}$ is given by Lemma \ref{obsbarriers}.  This follows through comparison on compact sets, using the barriers constructed in Lemma \ref{obsbarriers} with $\mathcal{S}=\{x\}$
(see the proof of Lemma \ref{obsbarriers} and Corollary \ref{cor:touch barriers}). From there, we can apply the monotone comparison principle in Theorem \ref{monocompthrm} to conclude uniqueness of the solution.
\end{proof}

\begin{proof}[Proof of Theorem \ref{obstacleregthrm}]
To prove the Lipschitz regularity it suffices to show that $W(x)=(U(x+y)-L_0|y|)\vee \Gamma^-(x)$ is a subsolution for any $y\in C^+$.
Indeed, recalling that both $\Gamma^+$ and $\Gamma^-$ are $L_0$-Lipschitz, this would imply
\begin{equation}
\label{eq:Lip}
U(x+y) \leq U(x)+L_0|y| \qquad \forall \,x \in \R^N, \, y \in C^+.
\end{equation}
Now, given any two points $x,z \in \R^N$, by a simple geometric construction one can always find a third point $w$ such that
$$
x-w,z-w \in C^+,\qquad |x-w|+|z-w| \leq A_\theta |x-z|,
$$
where $A_\theta\geq 1$ is a constant depending only on the opening of the cone $C^+$. (The point $w$ can be found as
the (unique) point of $(x+C^-)\cap(x+C^-)$ closest to both $x$ and $z$.)
So, by \eqref{eq:Lip} and the monotonicity of $U$ in directions of $C^+$ we get
$$
U(z) \leq U(w) + L_0|z-w| \leq U(x)+ A_\theta L_0  |x-z| \qquad \forall\, x,z\in \R^N,
$$
and the Lipschitz continuity follows.

The fact that $W \in \mathcal F$ follows as in many previous arguments by using that $U(x+y)-L_0h$ is a subsolution and that
$U(x+yh)-L_0h\leq\Gamma^+(x)$ for all $x\in\mathbb{R}^N$ (since $U \leq \Gamma^+$ and $\Gamma^+$ is $L_0$ Lipschitz).

We now examine how $U$ leaves the obstacle near the free boundary $\partial \{U=\Gamma^+\}$ (the case of $\Gamma^-$ is similar).
Consider a point $x\in\mathbb{R}^N$ such that $U(x)=\Gamma^+(x)$. Since $U$ solves \eqref{obstaclePDE} inside the set $\{U>\Gamma^-\}\supset \{U=\Gamma^+\}$, we know $\I U(x)\geq 0$.
Moreover, $\Gamma^+$ is a $C^{1,1}$ function touching $U$ from above at $x$, so $\I U(x)$ may be evaluated classically in the direction of $\nabla\Gamma^+(x)$
(see, for instance, \cite[Lemma 3.3]{MR2494809}).  Let $y\in S^{N-1}$ denote this direction. For any $r>0$, we have
\begin{align}
0&\leq\int_0^\infty\frac{U(x+\eta y)+U(x-\eta y)-2U(x)}{\eta^{1+2s}}\,d\eta\notag\\
&=\int_0^r \frac{U(x+\eta y)-\Gamma^+(x+\eta y)+U(x-\eta y)-\Gamma^+(x-\eta y)}{\eta^{1+2s}}\,dx\notag\\
&\ \ \ \ \ \ \ +\int_0^r\frac{\Gamma^+(x+\eta y)+\Gamma^+(x-\eta y)-2\Gamma^+(x)}{\eta^{1+2s}}\,d\eta\notag\\ 
&\ \ \ \ \ \ \ \ \ \ \ \ +\int_r^\infty\frac{U(x+\eta y)+U(x-\eta y)-2U(x)}{\eta^{1+2s}}\,d\eta.\notag
\end{align}
Since $U$ is bounded, $\Gamma^+$ is uniformly $C^{1,1}$ and monotone inside the set $\{\Gamma^+<1\}$, and $U\leq \Gamma^+$, we get
\begin{align}
0\leq \int_0^r \frac{\Gamma^+(x+\eta y)-U(x+\eta y)}{\eta^{1+2s}}\,dx \leq C\notag
\end{align}
for some constant $C<\infty$ independent of $r$.  Set $\delta= \Gamma^+(x+r y)-U(x+r y)$.
Since $\Gamma^+$ is $L_0$-Lipschitz, we have $\Gamma^+(x+\eta y)-U(x+ry)\geq \frac{\delta}{2}$ for all $\eta\in (r-\delta/(2L_0),r)$.  Then,
\begin{align}
\frac{\delta^2}{4L_0}\frac{1}{(r-\delta/(2L_0))^{1+2s}} &\leq \frac{\delta}{2} \int_{r-\delta/(2L_0)}^r \frac{1}{\eta^{1+2s}}\,d\eta \notag\\
&\leq \int_0^r \frac{\Gamma^+(x+\eta y)-U(x+\eta y)}{\eta^{1+2s}}\,dx\leq C.\notag
\end{align}
This implies $\delta^2\leq C'r^{1+2s}$ for some universal constant $C'$. Hence
\begin{align}
\Gamma^+(x+r y)-U(x+r y) =O(r^{1+(s-1/2)})\notag
\end{align}
as $r\rightarrow 0$, as desired.
\end{proof}

\section{Appendix}
\begin{figure}
\centerline{\epsfysize=4truein\epsfbox{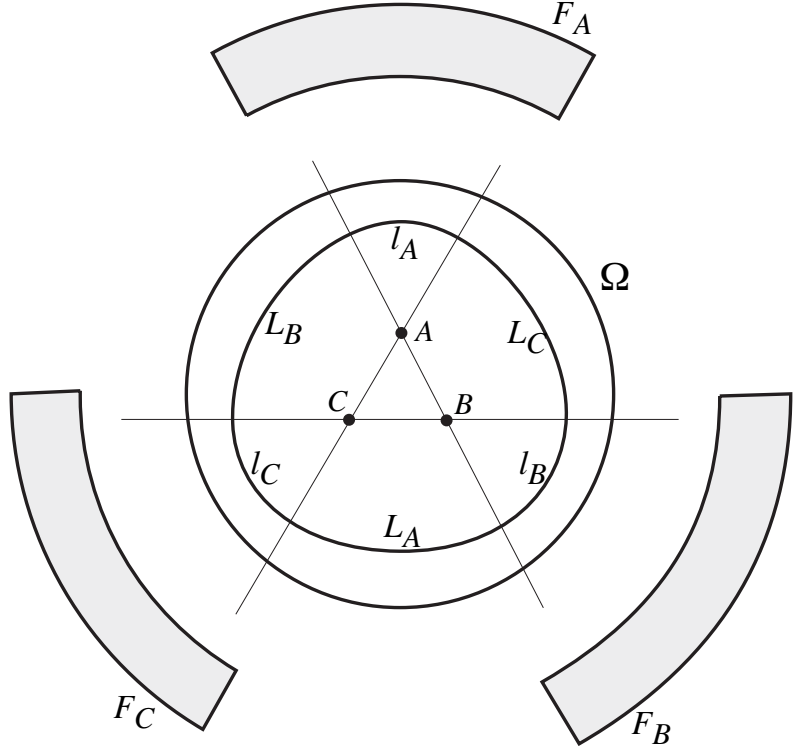}}
\caption{Geometry for non-uniqueness example.}
\label{figss}
\end{figure}
Figure \ref{figss} describes a simple bi-dimensional geometry for which $u\equiv 0$ is
a solution in the sense of
\eqref{altsub}-\eqref{altsup}, and we can also exhibit a positive subsolution.

In this figure, $A$, $B$, and $C$ form an equilateral triangle centered at the origin.
The curves $l_A$ and $L_A$ are both sections of a circle centered at $A$.
Likewise $l_B$, $L_B$, $l_C$ and $L_C$ are sections of a circle centered at $B$ and $C$ respectively.
The sets $F_A$, $F_B$, and $F_C$ are obtained by intersecting an annulus centered at the origin
with sectors of angle $\pi/3$, and they contain the support of the boundary data.
Notice that, for $\alpha=A,B,C$, any line perpendicular to either $L_\alpha$ or $l_\alpha$ passes through the interior of $F_\alpha$.
Finally, $\Omega$ is the unit disc centered at the origin.  

To define the data consider the smaller set $F_A^\delta= \{x\in F_A| d(x,\partial F_A)>\delta\}$.
For some small $\delta>0$ let $f_A$ be a smooth function equal to $1$ in
$F_A^{2\delta}$, and equal to $0$ in $\R^2 \setminus F_A^\delta$.  With a similar definition for $f_B$ and $f_C$, let $f=f_A+f_B+f_C$.  We choose $\delta$
small enough that any line perpendicular to $l_\alpha$ or $L_\alpha$ intersects $\{x|f_\alpha(x)=1\}$ ($\alpha=A,B,C$).  

Consider now the problem (\ref{PDE}).  We first note that $u\equiv 0$ is a solution in the sense of \eqref{altsub}-\eqref{altsup}.
Indeed, \eqref{altsub} trivially holds. Moreover, 
for any point $x\in \Omega$ there is a line passing through the point which does not intersect the support of $f$,
so also \eqref{altsup} is satisfied. We now construct a subsolution which is larger than this solution, showing that a comparison
principle cannot hold.

Let $S$ be the compact set whose boundary is given by the $C^{1,1}$ curve
$\bigcup_{\alpha\in \{A,B,C\}} \bigl(L_\alpha\cup l_\alpha\bigr)$.  For $\rho>0$ small, let
$\phi_\rho\in C_c^\infty(\mathbb{R})$ be such that $\phi(t)=0$ for $t\leq 0$ and $t=1$ for $t>\rho$.
Finally, define
\begin{align}
u(x)=\left\{\begin{array}{lrl}
\phi_\rho\bigl(d(x,\partial S)\bigr) &\text{if} &x\in S,\\ 
0&\text{if} &x\in S^C.
\end{array}\right.\notag
\end{align}

Since $\partial S$ is $C^{1,1}$, by choosing $\rho$ sufficiently small we can guarantee $u$ is $C^{1,1}$ as well.

Define now $u_\epsilon=\epsilon u$.
For any $x\in S$ such that $\nabla u_\epsilon(x)\neq 0$, the line with direction $\nabla u_\epsilon(x)$
passing through the point $x$ will intersect the set $\{x|f(x)=1\}$, and thus there will be a uniform
positive contribution to the integral defining the operator.
Analogously, for any point $x\in S$ such with $\nabla u(x)=0$
there is a line which intersects the set $\{x|f(x)=1\}$, so there will be again a uniform positive
contribution to the operator.

Hence, by choosing $\epsilon$ sufficiently small, we can guarantee
that this positive contribution outweighs any negative contribution coming from the local shape of $u_\epsilon$
(since this contribution will be of order $\epsilon$), and therefore $u$ is a subsolution both in the sense of
Definition \ref{subsupdefn} and of \eqref{altsub}-\eqref{altsup}.

\bibliographystyle{plain}
\bibliography{non-local-tug}

\begin{thebibliography}{10}

\bibitem{MR2341994}
E.~N. Barron, L.~C. Evans, and R.~Jensen.
\newblock The infinity {L}aplacian, {A}ronsson's equation and their
  generalizations.
\newblock {\em Trans. Amer. Math. Soc.}, 360(1):77--101, 2008.

\bibitem{MR1351007}
L.~A. Caffarelli and X.~Cabr{\'e}.
\newblock {\em Fully nonlinear elliptic equations}, volume~43 of {\em American
  Mathematical Society Colloquium Publications}.
\newblock American Mathematical Society, Providence, RI, 1995.

\bibitem{MR2367025}
L.~A. Caffarelli, S.~Salsa, and L.~Silvestre.
\newblock Regularity estimates for the solution and the free boundary of the
  obstacle problem for the fractional {L}aplacian.
\newblock {\em Invent. Math.}, 171(2):425--461, 2008.

\bibitem{MR2494809}
L.~A. Caffarelli and L.~Silvestre.
\newblock Regularity theory for fully nonlinear integro-differential equations.
\newblock {\em Comm. Pure Appl. Math.}, 62(5):597--638, 2009.

\bibitem{CLM}
A.~Chambolle, E.~Lindgren, and Monneau R.
\newblock The {H}\"older infinite {L}aplacian and {H}\"older extensions.
\newblock {\em Preprint.}, 2010.

\bibitem{MR2408259}
M.~G. Crandall.
\newblock A visit with the {$\infty$}-{L}aplace equation.
\newblock In {\em Calculus of variations and nonlinear partial differential
  equations}, volume 1927 of {\em Lecture Notes in Math.}, pages 75--122.
  Springer, Berlin, 2008.

\bibitem{MR1804769}
M.~G. Crandall and L.~C. Evans.
\newblock A remark on infinity harmonic functions.
\newblock In {\em Proceedings of the {USA}-{C}hile {W}orkshop on {N}onlinear
  {A}nalysis ({V}i\~na del {M}ar-{V}alparaiso, 2000)}, volume~6 of {\em
  Electron. J. Differ. Equ. Conf.}, pages 123--129 (electronic), San Marcos,
  TX, 2001. Southwest Texas State Univ.

\bibitem{MR1861094}
M.~G. Crandall, L.~C. Evans, and R.~F. Gariepy.
\newblock Optimal {L}ipschitz extensions and the infinity {L}aplacian.
\newblock {\em Calc. Var. Partial Differential Equations}, 13(2):123--139,
  2001.

\bibitem{MR1118699}
M.~G. Crandall, H.~Ishii, and P.-L. Lions.
\newblock User's guide to viscosity solutions of second order partial
  differential equations.
\newblock {\em Bull. Amer. Math. Soc. (N.S.)}, 27(1):1--67, 1992.

\bibitem{MR1218686}
Robert Jensen.
\newblock Uniqueness of {L}ipschitz extensions: minimizing the sup norm of the
  gradient.
\newblock {\em Arch. Rational Mech. Anal.}, 123(1):51--74, 1993.

\bibitem{MR2200259}
R.~V. Kohn and S.~Serfaty.
\newblock A deterministic-control-based approach to motion by curvature.
\newblock {\em Comm. Pure Appl. Math.}, 59(3):344--407, 2006.

\bibitem{MR2449057}
Y.~Peres, O.~Schramm, S.~Sheffield, and D.~B. Wilson.
\newblock Tug-of-war and the infinity {L}aplacian.
\newblock {\em J. Amer. Math. Soc.}, 22(1):167--210, 2009.

\end{thebibliography}
\end{document}